\documentclass[reqno]{amsart}
\usepackage{amsmath,amssymb,color}

\usepackage[notcite,notref]{showkeys}

\newtheorem{theorem}{Theorem}[section]
\newtheorem{lemma}[theorem]{Lemma}

\newtheorem{corollary}[theorem]{Corollary}

\numberwithin{equation}{section}

\newcommand{\mc}[1]{{\mathcal #1}}
\newcommand{\mb}[1]{{\mathbf #1}}
\newcommand{\bb}[1]{{\mathbb #1}}
\newcommand{\<}{\langle}
\renewcommand{\>}{\rangle}
\newcommand{\asy}{\mathrm{a}}
\renewcommand{\epsilon}{\varepsilon}

\newcommand{\upbar}[1]{\,\overline{\! #1}}

\definecolor{light}{gray}{.9}

\begin{document}

\title[Weakly asymmetric exclusion process] {Strong
  asymmetric limit of the quasi-potential of the boundary driven 
  weakly asymmetric exclusion process}

\author [L. Bertini] {Lorenzo Bertini}
\address{\noindent Lorenzo Bertini \hfill\break\indent 
Dipartimento di Matematica, Universit\`a di Roma `La Sapienza' 
\hfill\break\indent 
P.le Aldo Moro 2, 00185 Roma, Italy}
\email{bertini@mat.uniroma1.it}

\author[D. Gabrielli]{Davide Gabrielli}
\address{Davide Gabrielli  \hfill\break\indent 
Dipartimento di Matematica, Universit\`a dell'Aquila  \hfill\break\indent 
67100 Coppito, L'Aquila, Italy}
\email{gabriell@univaq.it}

\author[C. Landim]{Claudio Landim}
\address{Claudio Landim \hfill\break\indent 
IMPA \hfill\break\indent 
Estrada Dona Castorina 110, \hfill\break\indent
J. Botanico, 22460 Rio
de Janeiro, Brazil\hfill\break\indent 
{\normalfont and} \hfill\break\indent 
CNRS UMR 6085, Universit\'e de Rouen, \hfill\break\indent 
Avenue de l'Universit\'e, BP.12,
Technop\^ole du Madril\-let, \hfill\break\indent
F76801 Saint-\'Etienne-du-Rouvray, France.} 
\email{landim@impa.br}


\noindent 


\begin{abstract}
  We consider the weakly asymmetric exclusion process on a bound\-ed
  interval with particles reservoirs at the endpoints. The
  hydrodynamic limit for the empirical density, obtained in the
  diffusive scaling, is given by the viscous Burgers equation with
  Dirichlet boundary conditions. In the case in which the bulk
  asymmetry is in the same direction as the drift due to the boundary
  reservoirs, we prove that the quasi-potential can be expressed in
  terms of the solution to a one-dimensional boundary value problem
  which has been introduced by Enaud and Derrida \cite{de}. We
  consider the strong asymmetric limit of the quasi-potential and
  recover the functional derived by Derrida, Lebowitz, and Speer
  \cite{DLS3} for the asymmetric exclusion process.
\end{abstract}

\maketitle
\thispagestyle{empty}

\section{Introduction}
\label{sec1}

The study of steady states of non-equilibrium systems has motivated a
lot of works over the last decades. It is now well established that
the steady states of non-equilibrium systems exhibit in general
long-range correlations and that the thermodynamic functionals, such
as the free energy, are not local nor additive.

The analysis of the large deviations asymptotics of stochastic
lattice gases with particle reservoirs at the boundary has proven
itself to be an important step in the physical description of
\emph{nonequilibrium stationary states} and a rich source of
mathematical problems. We refer to \cite{BDGJL7,Der} for two recent
reviews on this topic.

We consider a boundary driven one-dimensional lattice gas whose
dynamics can be informally described as follows.  Fix an integer $N\ge
1$, an external force $E$ in $\bb R$ and boundary densities $0<\rho_-
<\rho_+<1$. At any given time each site of the interval $\{-N+1,
\dots, N-1\}$ is either empty or occupied by one particle.  In the
bulk, each particle attempts to jump to the right at rate $ 1+ E/2N$
and to the left at rate $1 - E/2N$. To respect the exclusion rule, the
particle jumps only if the target site is empty, otherwise nothing
happens. At the boundary sites $\pm (N-1)$ particles are created and
removed for the local density to be $\rho_\pm$: at rate $\rho_\pm$ a
particle is created at $\pm (N-1)$ if the site is empty and at rate
$1-\rho_\pm$ the particle at $\pm (N-1)$ is removed if the site is
occupied.

The dynamics just described defines an irreducible Markov process on a
finite state space which has a unique stationary state denoted by
$\mu^N_E$. Let $\varphi_\pm := \log [ \rho_\pm/ (1-\rho_\pm)]$ be the
chemical potential of the boundary reservoirs and set
$E_0:=(\varphi_+-\varphi_-)/2$. When $E=E_0$, the drift caused by the
external field $E$ matches the drift due to the boundary reservoirs,
and the process becomes reversible.

In the limit $N\uparrow\infty$, the typical density profile
$\upbar{\rho}_E$ under the stationary state $\mu^N_E$ can be described
as follows. For each $E\le E_0$ there exists a unique $J_E\le 0$ such
that
\begin{equation*}
\frac 12 \int_{\rho_-}^{\rho_+}\! dr \: \frac{1}{E\chi(r)-J_E} = 1\;,
\end{equation*}
where $\chi$ is the mobility of the system: $\chi(a) = a(1-a)$. The
profile $\upbar{\rho}_E$ is then obtained by solving
\begin{equation*}
\upbar{\rho}'_E - E \, \chi( \upbar{\rho}_E ) = - J_E
\end{equation*}
with the boundary condition $\upbar{\rho}_E(-1)=\rho_-$.  

In the same limit $N\uparrow\infty$, the probability of observing a
density profile $\gamma$ different from $\upbar{\rho}_E$ can be
expressed as
\begin{equation}
\label{aa1}
\mu^N_E \{\gamma\} \;\sim\; \exp\{ -N V_E (\gamma)\}\;.
\end{equation}
The large deviations functional $V_E$, which also depends on $\rho_-$,
$\rho_+$, is an extension of the notion of free energy to the context
of non-equilibrium systems. 

The free energy of a boundary driven lattice gas has first been 
derived for the symmetric simple exclusion process by Derrida,
Lebowitz and Speer \cite{DLS3} based on the so called matrix method,
introduced by Derrida, which permits to express the stationary state
$\mu^N_E$ as a product of matrices.  Bertini et al.  \cite{bdgjl2}
derived the same result through a dynamical approach which we extend
here to the weakly asymmetric case.

We consider only the situation $E<E_0$ for the bulk asymmetry to be in
the same direction as the drift due to the boundary. The reversible
case $E=E_0$ lacks interest because the stationary state is product
and does not exhibit long range correlations. In contrast, the
analysis of the quasi-potential $V_E$ for $E>E_0$, not treated here,
appears a most interesting problem. For instance, a representation of
$V_E$ as a supremum of trial functionals analogous to \eqref{ssg} below
seems to be ruled out.

In the boundary driven weakly asymmetric exclusion process, for $E<
E_0$, the quasi-potential takes the following form:
\begin{eqnarray}
\label{aa2}
V_E(\gamma)  &:=&
\int_{-1}^1\!du \: \Big\{ \gamma\log\gamma + (1-\gamma)\log(1-\gamma)
+(1-\gamma)\varphi -\log\big(1+e^{\varphi}\big) 
\nonumber
\\ && \qquad \phantom{\int_{-1}^1\!du \: \Big\{}
+ \; \frac 1E \Big[ \varphi' \log \varphi' 
- (\varphi' - E) \log (\varphi'  - E) \Big] - A_E 
\Big\}\;,\qquad\quad
\end{eqnarray}
where $A_E$ is the constant given by
\begin{equation*}
A_E\; :=\; \log(-J_E) \;+\; \frac 12 \int_{\gamma_-}^{\gamma_+} \!dr \, 
\frac 1{E \chi(r)}
\log \Big[ 1 -  \frac {E \chi(r)}{J_E} \Big] \;;
\end{equation*}
and where $\varphi$ is the unique solution of the Euler-Lagrange
equation
\begin{equation*}
\frac{\varphi''}{\varphi' (\varphi' -E)}
\;+\; \frac 1{1+ e^{\varphi}} \;=\; \gamma 
\end{equation*}
satisfying $\varphi(\pm 1) =\varphi_\pm$, $\varphi' > \max\{0, E\}$.

This result, stated in a different form, has been proved by Enaud and
Derrida \cite{de} based on the matrix method. We prove this result in
Section \ref{s:6} below by the dynamical approach introduced in
\cite{bdgjl2, bdgjl4}. We also show that the quasi-potential is convex
and lower semi-continuous.

In section \ref{s:asyl}, we show that $V_E$ $\Gamma$-converges, as
$E\downarrow -\infty$, to the free energy of the boundary driven
asymmetric exclusion process, first derived by Derrida, Lebowitz and
Speer \cite{DLS3}. This asymptotic behavior is somewhat surprising
since the hydrodynamic time scales at which the weakly asymmetric
exclusion process and the asymmetric exclusion process evolve are
different. We also prove convergence of the solutions of the
Euler-Lagrange equations as the external force $E$ diverges.

The dynamical approach followed here permits to compute the
fluctuation probabilities \eqref{aa1} in great generality, in any
dimension and for a large class of processes. However, it is only in
dimension one and for very few interacting particle systems that an
explicit expression of type \eqref{aa2} is available for the
non-equilibrium free energy $V_E$.

\section{Notation and Results}
\label{sec2}

\subsection*{ The boundary driven  weakly asymmetric exclusion process}

Fix an integer $N\geq 1$, $E \in \bb R$, $0<\rho_- \leq \rho_+ <
1$ and let $\Lambda_N :=\{-N+1, \dots , N-1\}$. The configuration space
is $\Sigma_N:=\{0,1\}^{\Lambda_N}$; elements of $\Sigma_N$ are denoted
by $\eta$ so that $\eta(x)=1$, resp.\ $0$, if site $x$ is occupied,
resp.\ empty, for  the configuration $\eta$.
We denote by $\sigma^{x,y}\eta$ the configuration obtained from $\eta$ by
exchanging the occupation variables $\eta (x)$ and $\eta (y)$, i.e.\ 
\begin{equation*}
(\sigma^{x,y} \eta) (z) := 
  \begin{cases}
        \eta (y) & \textrm{ if \ } z=x\\
        \eta (x) & \textrm{ if \ } z=y\\
        \eta (z) & \textrm{ if \ } z\neq x,y,
  \end{cases}
\end{equation*}
and by $\sigma^{x}\eta$ the configuration obtained from
$\eta$ by flipping the configuration at $x$, i.e.\
\begin{equation*}
(\sigma^{x} \eta) (z) := 
  \begin{cases}
        1-\eta (x) & \textrm{ if \ } z=x\\
        \eta (z) & \textrm{ if \ } z\neq x.
  \end{cases}
\end{equation*}

The one-dimensional boundary driven weakly asymmetric exclusion
process is the Markov process on $\Sigma_N$ whose generator $L_N$ 
can be decomposed as
\begin{equation}
\label{f03}
L_N \;=\; L_{0,N} \;+\; L_{-,N} \;+\; L_{+,N} \;,
\end{equation}
where the generators $L_{0,N}$, $L_{-,N}$, $L_{+,N}$ act on functions
$f:\Sigma_N\to \bb R$ as
\begin{eqnarray*}
&& (L_{0,N} f)(\eta) \;=\; \frac{N^2}2 \sum_{x=-N+1}^{N-2} 
e^{- E/(2N) \, [ \eta(x+1)-\eta(x)]}  
\big[ f(\sigma^{x,x+1} \eta)-f(\eta)\big] \;, \\
&& \quad
(L_{-,N} f)(\eta) \;=\; \frac{N^2}2 \: 
c_- \big( \eta(-N+1) \big) \big[ f(\sigma^{-N+1} \eta)-f(\eta)\big] 
\\
&&  \qquad (L_{+,N} f)(\eta) \;=\; \frac{N^2}2 \: 
 c_+ \big( \eta(N-1) \big) \big[ f(\sigma^{N-1} \eta)-f(\eta)\big] 
\end{eqnarray*}
where $c_\pm: \{0,1\}\to \bb R$ are given by
\begin{equation*}
  c_\pm (\zeta) \;:=\;
\rho_\pm e^{\mp E/(2N)}(1-\zeta) + (1-\rho_\pm) e^{ \pm E/(2N)}
\zeta\;. 
\end{equation*}
Notice that the (weak) external field is $E/(2N)$ and, in view of the
diffusive scaling limit, the generator has been speeded up by $N^2$.
We denote by $\eta_t$ the Markov process on $\Sigma_N$ with generator
$L_N$ and by $\bb P^N_\eta$ its distribution if the initial
configuration is $\eta$.  Note that $\bb P^N_\eta$ is a probability
measure on the path space $D(\bb R_+, \Sigma_N)$, which we consider
endowed with the Skorohod topology and the corresponding Borel
$\sigma$-algebra. Expectation with respect to $\bb P^N_\eta$ is
denoted by $\bb E^N_\eta$.

Since the Markov process $\eta_t$ is irreducible, for each $N\ge 1$,
$E\in\bb R$, and $0<\rho_-\le \rho_+<1$ there exists a unique invariant
measure $\mu^N_E$ in which we drop the dependence on $\rho_\pm$ from
the notation.  Let $\varphi_\pm := \log [ \rho_\pm/ (1-\rho_\pm)]$ be
the chemical potential of the boundary reservoirs and set
$E_0:=(\varphi_+-\varphi_-)/2$. A simple computation shows that if
$E=E_0$ then the process $\eta_t$ is reversible with respect to the
product measure
\begin{equation}
\label{imr}
  \mu^N_{E_0}(\eta) = \prod_{x=-N+1}^{N-1} 
\frac{ e^{\upbar{\varphi}^N_{E_0}
  (x) \, \eta(x)}}{1+e^{\upbar{\varphi}_{E_0}^N(x)}}
\end{equation}
where 
\begin{equation*}
 \upbar{\varphi}^N_{E_0}(x) := \varphi_- \frac{ N-x}{2N} 
 + \varphi_+ \frac{ N+x}{2N}
\end{equation*}
On the other hand, for $E \neq E_0$ the invariant measure $\mu^N_E$
cannot be written in a simple form. 

\subsection*{The dynamical large deviation principle}

We denote by $u\in [-1,1]$ the macroscopic space coordinate and by 
$\langle \cdot,\cdot \rangle$ the inner product in
$L_2\big([-1,1],du\big)$. We set
\begin{equation}
\label{dcm}
\mc M := \left\{ \rho \in L_\infty \big([-1,1],du\big) \,:\:
0\leq \rho \leq 1 \right\}
\end{equation}
which we equip with the topology induced by the weak convergence of
measures, namely a sequence $\{\rho^n\} \subset \mc M$ converges to
$\rho$ in $\mc M$ if and only if $\langle \rho^n, G \rangle \to
\langle \rho, G \rangle$ for any continuous function $G: [-1,1]\to\bb
R$.  Note that $\mc M$ is a compact Polish space that we consider
endowed with the corresponding Borel $\sigma$-algebra.
The empirical density of the configuration $\eta\in\Sigma_N$ is
defined as $\pi^N(\eta)$ where the map $\pi^N \colon \Sigma_N \to \mc
M$ is given by  
\begin{equation}
\label{eq:2}
\pi^N (\eta) \, (u) \;:=\; 
\sum_{x=-N+1}^{N-1} \eta (x) \,
\mb 1 \{ \big[ \frac{x}{N}- \frac{1}{2N}, 
\frac{x}{N}+ \frac{1}{2N}\big)\} (u) \; , 
\end{equation}
in which $\mb 1\{A\}$ stands for the indicator function of the set
$A$.  Let $\{\eta^N\}$ be a sequence of configurations with
$\eta^N\in\Sigma_N$. If the sequence $\{\pi^N(\eta^N)\}\subset\mc M$
converges to $\rho$ in $\mc M$ as $N\to\infty$, we say that
$\{\eta^N\}$ is \emph{associated} the macroscopic density profile
$\rho\in\mc M$.

Given $T>0$, we denote by $D\big([0,T];\mc M\big)$ the Skorohod space
of paths from $[0,T]$ to $\mc M$ equipped with its Borel
$\sigma$-algebra. Elements of $D\big([0,T], \mc M\big)$ will be
denoted by $\pi\equiv\pi_t(u)$ and sometimes by $\pi(t,u)$.  Note that
the evaluation map $D\big([0,T];\mc M\big) \ni \pi \mapsto \pi_t \in
\mc M$ is not continuous for $t\in (0,T)$ but is continuous for
$t=0,T$.  We denote by $\pi^N$ also the map from
$D\big([0,T];\Sigma_N\big)$ to $D\big([0,T];\mc M\big)$ defined by
$\pi^N(\eta_\cdot)_t := \pi^N(\eta_t)$. The notation $\pi^N(t,u)$ is
also used.

Fix a profile $\gamma\in \mc M$ and consider a sequence $\{\eta^N :
N\ge 1\}$ associated to $\gamma$. Let $\eta_t^N$ be the boundary
driven weakly asymmetric exclusion process starting from $\eta^N$.  In
\cite{dps,g,kov} it is proven that as $N\to\infty$ the sequence of
random variables $\{\pi^N(\eta^N_\cdot)\}$, which take values in $\mc
D\big([0,T];\mc M\big)$, converges in probability to the path
$\rho\equiv\rho_t(u)$, $(t,u)\in [0,T] \times[-1,1]$ which solves the
viscous Burgers equation with Dirichlet boundary conditions at $\pm
1$, i.e.\ 
\begin{equation} 
\label{eq:1}
\begin{cases}
{\displaystyle
\partial_t \rho 
+ \frac E2 \,  \nabla \chi(\rho)
= \frac 12  \, \Delta \rho } \\
{\displaystyle
\vphantom{\Big\{}
\rho_t(\pm 1) = \rho_\pm
} 
\\
{\displaystyle
\rho_0(u) = \gamma(u) 
}
\end{cases}
\end{equation}
where $\chi:[0,1]\to\bb R_+$ is the mobility of the system, $\chi(a) =
a(1-a)$, and $\nabla$, resp.\ $\Delta$, denotes the derivative, resp.\ 
the second derivative, with respect to $u$.  In fact the proof
presented in \cite{dps,g} is in real line, while the one in \cite{kov}
is on the torus. The arguments however can be adapted to the boundary
driven case, see \cite{els1,els2,klo2} for the hydrodynamic limit of
different boundary driven models.

A large deviation principle for the empirical density can also be
proven following \cite{kov, q, qrv}, adapted to the open boundary
context in \cite{bdgjl4}. In order to state this result some more
notation is required.  Fix $T>0$ and let $\Omega_T = (0,T)\times
(-1,1)$, $\overline{\Omega_T} = [0,T]\times [-1,1]$. For positive
integers $m,n$, we denote by $C^{m,n}(\overline{\Omega_T})$ the space
of functions $G\equiv G_t(u)\colon \overline{\Omega_T}\to\bb R$ with
$m$ derivatives in time, $n$ derivatives in space which are continuous
up the the boundary. We improperly denote by $C^{m,n}_0(\overline
{\Omega_T})$ the subset of $C^{m,n}(\overline {\Omega_T})$ of the
functions which vanish at the endpoints of $[-1,1]$, i.e.\ $G\in
C^{m,n}(\overline{\Omega_T})$ belongs to
$C^{m,n}_0(\overline{\Omega_T})$ if and only if $G_t(\pm 1)=0$,
$t\in[0,T]$.

Let the energy $\mc Q: D([0,T], \mc M) \to [0,\infty]$ be given by
\begin{eqnarray*}
\!\!\!\!\!\!\!\!\!\!\!\!\! &&
\mc Q(\pi) \;=\;  \\
\!\!\!\!\!\!\!\!\!\!\!\!\! && \quad
\sup_{H} \Big\{ \int_0^Tdt \int_{-1}^1 du\,  \pi(t,u) \,
(\nabla H)(t,u) \;-\; \frac 12  \int_0^Tdt \int_{-1}^1 du\,
H(t,u)^2 \, \chi (\pi(t,u)) \Big\}\;, 
\end{eqnarray*}
where the supremum is carried over all smooth functions $H: \Omega_T
\to\bb R$ with compact support. If $\mc Q(\pi)$ is finite, $\pi$ has a
generalized space derivative, $\nabla \pi$, and
\begin{equation*}
\mc Q(\pi) \;=\; \frac 12 \int_0^T dt\, \int_{-1}^1 du\,
\frac{(\nabla \pi_t)^2}{\chi (\pi_t)}\;\cdot 
\end{equation*}

Fix a function $\gamma\in\mc M$ which corresponds to the initial
profile.  For each $H$ in $C^{1,2}_0(\overline {\Omega_T})$, let $\hat
J_{H} = \hat J_{T, H, \gamma} \colon D([0,T], \mc M)\longrightarrow
\bb R$ be the functional given by
\begin{eqnarray*}
\hat J_H(\pi) &:=& \big\langle \pi_T, H_T \big\rangle 
- \langle \gamma, H_0 \rangle
- \int_0^{T} \!dt\, \big\langle \pi_t, \partial_t H_t \big\rangle
\\
&-& \frac 12 \int_0^{T} \!dt\, \big\langle \pi_t , \Delta H_t \big\rangle
\;+\; \frac {\rho_+}2   \int_0^{T} \! dt\, \nabla H_t(1) 
\; -\;  \frac {\rho_-}2  \int_0^{T} \!dt\, \nabla H_t(-1) \\
&-& \frac E2 \int_0^{T} \!dt\,  
\big\langle \chi(\pi_t), \nabla H_t \big\rangle 
\;-\; \frac 12 \int_0^{T} \!dt\,
\big\langle \chi( \pi_t ), \big( \nabla H_t \big)^2 \big\rangle \; .
\end{eqnarray*}
Let $\hat I_T (\, \cdot \, | \gamma) \colon D([0,T],\mc
M)\longrightarrow [0,+\infty]$ be the functional defined by
\begin{equation}
\label{f01}
\hat I_T (\pi | \gamma) \; :=\; \sup_{H\in
  C^{1,2}_0(\overline {\Omega_T})} \hat J_H(\pi)\; .
\end{equation}
The rate functional $I_T(\cdot | \gamma): D([0,T], \mc M) \to
[0,\infty]$ is given by
\begin{equation}
\label{f02}
I_T(\pi | \gamma) \;=\; 
\left\{
\begin{array}{ll}
\hat I_T (\pi | \gamma) & \text{if $\mc Q(\pi) < \infty$ \;,} \\
\infty & \text{otherwise.}
\end{array}
\right.
\end{equation}
Is is proved in \cite{blm1}, for any $E$ in $\bb R$, that the
functional $I_T (\cdot | \gamma)$ is lower semicontinous, has compact
level sets and that a large deviations principle for the empirical
measure holds.

\begin{theorem}
\label{s02}
Fix $T>0$ and an initial profile $\gamma$ in $\mc M$.  Consider a
sequence $\{\eta^N : N\ge 1\}$ of configurations associated to
$\gamma$.  Then, the sequence of probability measures $\{\bb
P^N_{\eta^N} \circ (\pi^N)^{-1} : N\ge 1\}$ on $D([0,T];\mc M)$
satisfies a large deviation principle with speed $N$ and good rate
function $I_T(\cdot|\gamma)$.  Namely, $I_T(\cdot|\gamma):
D\big([0,T]; \mc M\big) \to [0,\infty]$ has compact level sets and for
each closed set $\mc C \subset D([0,T]; \mc M)$ and each open set $\mc
O \subset D([0,T]; \mc M)$
\begin{eqnarray*}
&& 
\varlimsup_{N\to\infty} \frac 1N \log \bb P^N_{\eta^N} 
\big( \pi^N \in \mc C\big)
\;\leq\; - \inf_{\pi \in \mc C} I_T (\pi | \gamma)  
\\
&& \qquad 
\varliminf_{N\to\infty} \frac 1N \log \bb P^N_{\eta^N} 
\big( \pi^N\in \mc O \big) \;\geq\; -  \inf_{\pi \in \mc O} I_T (\pi |
\gamma) \; . 
\end{eqnarray*}
\end{theorem}

\subsection*{The quasi-potential}

From now on we consider only the case $E\le E_0 = (\varphi_+ -
\varphi_-)/2$, where $\varphi_\pm =\log[\rho_\pm/(1-\rho_\pm)]$.
Simple computations, which are omitted, show that the unique
stationary solution $\upbar{\rho}_E \in \mc M$ of the hydrodynamic
equation \eqref{eq:1} can be described as follows.  For each $E\le
E_0$ there exists a unique $J_E\le 0$ such that
\begin{equation}
  \label{css}
  \frac 12 \int_{\rho_-}^{\rho_+}\! dr \: \frac{1}{E\chi(r)-J_E} = 1\;.
\end{equation}
The profile $\upbar{\rho}_E$ is then obtained by solving 
\begin{equation}
  \label{ss}
  \upbar{\rho}'_E - E \, \chi( \upbar{\rho}_E ) = - J_E
\end{equation}
with the boundary condition $\upbar{\rho}_E(-1)=\rho_-$.  Note that
$J_E/2$ is the current maintained by the stationary profile
$\upbar{\rho}_E$. The solution to \eqref{ss} can easily be written in
an explicit form, see \cite{de}.  We shall however only use, as can be
easily checked, that $\upbar\rho_E$ is strictly increasing and that
the inequality $J_E/E > \max_{r\in[\rho_-,\rho_+]} \chi(r)$ holds for
$E<0$.

Given $E\le E_0$, the \emph{quasi-potential} for the rate function
$I_T$ is the functional $V_E\,: \mc M \to \bb [0,+\infty]$ defined by 
\begin{equation}
\label{qp}
V_E(\rho) \;:=\; \inf_{T>0}\; \inf \big\{ 
I_T\big(\pi | \upbar{\rho}_E \big) \,,\: 
\pi\in D\big([0,T];\mc M\big) \,:\: \pi_T =\rho \big\} 
\end{equation}
so that $V_E(\rho)$ measures the minimal cost to produce the profile
$\rho$ starting from $\upbar\rho_E$.

Recall that $\mu^N_E$ is the unique invariant measure of the boundary
driven weakly asymmetric exclusion process. The following result,
which states that the quasi-potential gives the rate function of the
empirical density when particles are distributed according to
$\mu^N_E$ is proven in \cite{BG} in the case $E=0$. Thanks to
Theorem~\ref{s02}, the proof applies also to the weakly asymmetric
case.

\begin{theorem}
\label{t:rfim}
For each $E$ in $\bb R$, the sequence of probability measures on $\mc
M$ given by $\{\bb \mu^N_{E} \circ (\pi^N)^{-1}\}$ satisfies a large
deviation principle with speed $N$ and rate function $V_E$.  Namely,
for each closed set $\mc C \subset \mc M$ and each open set $\mc
O\subset \mc M$,
\begin{eqnarray*}
&& \varlimsup_{N\to\infty} \frac 1N \log \mu^N_{E} 
\big( \pi^N \in \mc C \big)
\;\leq\; - \inf_{\rho \in \mc C} V_E(\rho)  \\
&&\qquad 
\varliminf_{N\to\infty} \frac 1N \log \mu^N_{E} 
\big( \pi^N\in \mc O\big) 
\;\geq\; -  \inf_{\rho \in \mc O} V_E (\rho) \; .
\end{eqnarray*}
\end{theorem}

In this paper we prove that the quasi-potential $V_E$ can be expressed
in terms of the solution to a one-dimensional boundary value problem.
This result has been obtained in \cite{de} by analyzing directly the
invariant measure $\mu^N_E$ through combinatorial techniques; while we
here follow instead the dynamic approach \cite{bdgjl2,bdgjl4} by
characterizing the optimal path for the variational problem
\eqref{qp}.

For $E<E_0$, let $C^{1+1}([-1,1])$ be the set of continuously
differentiable functions on $[-1,1]$ with Lipshitz derivative and set
\begin{equation}
\label{mcfe}
\mc F_E := \big\{ \varphi \in C^{1+1}([-1,1]) \,:\: 
\varphi(\pm 1) =\varphi_\pm\,,\: \varphi' > 0\lor E \big\}\;,
\end{equation}
where, given $a,b\in\bb R$, the notation $a\lor b$, resp.\ $a\wedge
b$, stands for $\max\{a,b\}$, resp.\ $\min\{a,b\}$. Note that $\mc F_E
= \mc F_{E'}$ for $E, E'<0$.

For $E< E_0$, $E\not =0$, let $\mc G_E : \mc M \times \mc F_E \to \bb
R$ be given by 
\begin{eqnarray}
\label{Grf}
\mc G_E  (\rho,\varphi) &:=&
\int_{-1}^1\!du \: \Big\{ \rho\log\rho + (1-\rho)\log(1-\rho)
+(1-\rho)\varphi -\log\big(1+e^{\varphi}\big) 
\nonumber
\\ && \phantom{\int_{-1}^1\!du \: \Big\{}
+ \; \frac 1E \Big[ \varphi' \log \varphi' 
- (\varphi' - E) \log (\varphi'  - E) \Big] - A_E 
\Big\}\;,\qquad\quad
\end{eqnarray}
where, by convention, $0\log 0 =0$ and $A_E$ is the constant given by
\begin{equation}
\label{cge}
A_E\; :=\; \log(-J_E) \;+\; \frac 12 \int_{\rho_-}^{\rho_+} \!dr \, 
\frac 1{E \chi(r)}
\log \Big[ 1 -  \frac {E \chi(r)}{J_E} \Big] \;.
\end{equation}
The right hand side is well defined because $J_E < 0$ and
$J_E/E>\max_{r\in[\rho_-,\rho_+]}\chi(r)$ for $E<0$.

For $E=0$, $\mc G_0: \mc M \times \mc F_E \to \bb R$ is defined by
continuity as
\begin{eqnarray*}
\!\!\!\!\!\!\!\!\!\!\! &&
\mc G_0  (\rho,\varphi) \;= \\
\!\!\!\!\!\!\!\!\!\!\! &&
\int_{-1}^1\!du \: \Big\{ \rho\log\rho + (1-\rho)\log(1-\rho)
+(1-\rho)\varphi  -\log\big(1+e^{\varphi}\big) 
+\; \log \varphi' +1 -A_0 \Big\}\;,
\end{eqnarray*}
where $A_0 = \log [(\rho_+ -\rho_-)/2]+1$.

For $E< E_0$, define the functional $S_E : \mc M \to \bb R$ by
\begin{equation}
\label{ssg}
S_E(\rho) \; :=\; \sup_{\varphi \in \mc F_E}  \: \mc G_E(\rho, \varphi)\;.
\end{equation}
Note that $S_E$ is a positive functional because a simple computation
relying on \eqref{ss} shows that
\begin{equation}
\label{diseg}
S_{E} (\rho) \;\ge\; \mc G_E(\rho, \upbar{\varphi}_{E})
\;=\; \int_{-1}^1 \!du\: \Big\{ \rho\log\frac{\rho}{\upbar\rho_{E}}    
+ (1-\rho)\log \frac{1-\rho}{1-\upbar\rho_{E}} \Big\}
\end{equation}
if $\upbar{\varphi}_{E}:= \log [\upbar\rho_E/(1-\upbar\rho_E)]$. 

In the special case $E=E_0$, as already observed, the weakly
asymmetric exclusion process is reversible and the stationary state
$\mu^N_{E_0}$ is a product measure. In particular, the rate functional
$S_{E_0}$ of the static large deviations principle for the empirical
density can be explicitly computed. It is given by
\begin{equation}
\label{f07}
S_{E_0} (\rho) 
\;=\; \int_{-1}^1 \!du\: \Big\{ \rho\log\frac{\rho}{\upbar\rho_{E_0}}    
+ (1-\rho)\log \frac{1-\rho}{1-\upbar\rho_{E_0}} \Big\} \;.
\end{equation}

The Euler-Lagrange equation associated to the variational problem
\eqref{ssg} is
\begin{equation}
\label{Deq}
\frac{\varphi''}{\varphi' (\varphi' -E)}
\;+\; \frac 1{1+ e^{\varphi}} \;=\; \rho \;.
\end{equation}
A function $\varphi\in\mc F_E$ solves the above equation when it is
satisfied Lebesgue a.e.  Recalling that the stationary profile
$\upbar\rho_E$ satisfies \eqref{ss} and \eqref{css}, it is easy to
check that if $\rho=\upbar\rho_E$ then $\upbar\varphi_E$ solves
\eqref{Deq} and $\mc G_E(\upbar\rho_E,\upbar\varphi_E)=0$.

The analysis of the quasi-potential for the boundary driven symmetric
exclusion process, i.e.\ the case $E=0$ of the current setting, has
been considered in \cite{bdgjl4}. In particular it is there shown that
$V_0$ coincides with $S_0$.  We prove in this article an analogous
statement for any $E\le E_0$.

\begin{theorem}
\label{t:S=V}
Let $E\le E_0$ and $V_E,S_E : \mc M \to [0,+\infty]$ be the
functionals defined in \eqref{qp}, \eqref{ssg} and \eqref{f07}.
\begin{enumerate}
\item[(i)] The functional $S_E$ is bounded, convex, and lower
  semicontinuous on $\mc M$.
\item[(ii)] Fix $E<E_0$. For each $\rho\in \mc M$ there exists in $\mc
  F_E$ a unique solution to \eqref{Deq} denoted by $\Phi(\rho)$.
  Moreover
\begin{equation}
\label{eq:S=V:2}
S_E(\rho) = \max_{\varphi\in\mc F_E} \mc G_E (\rho,\varphi) 
=\mc G_E (\rho,\Phi(\rho)) \;.
\end{equation}
\item[(iii)]
The equality $V_E=S_E$ holds on $\mc M$.
\end{enumerate}
\end{theorem}

The proof of the last item of the previous theorem is achieved by
characterizing the optimal path for the variational problem \eqref{qp}
defining the quasi-potential. For $E<E_0$ it is obtained by the
following algorithm. Given $\rho\in\mc M$ let $\Phi(\rho)\in\mc F_E$
be the solution to \eqref{Deq} and define $G=e^{\Phi(\rho)} /
[1+e^{\Phi(\rho)}]$.  Let $F\equiv F_t(u)$ be the solution to the
viscous Burgers equation \eqref{eq:1} with initial condition $G$ and
set $\psi = \log [ F/(1-F)]$, note that $\psi_0= \Phi(\rho)$ and
$\psi_t\to\upbar\varphi_E$ as $t\to\infty$.  Let $\rho^*_t =
\Phi^{-1}(\psi_t)$, i.e.\ $\rho^*_t$ is given by the l.h.s.\ of
\eqref{Deq} with $\varphi$ replaced by $\psi_t$. Observe that
$\rho^*_0=\rho$ and $\rho^*_t\to \upbar\rho_E$ as $t\to\infty$.  The
optimal path for \eqref{qp} is then $\pi^*_t=\rho^*_{-t}$, the fact
that it is defined on the time interval $(-\infty,0]$ instead of
$[0,\infty)$ makes no real difference. As discussed in
\cite{bdgjl2,BDGJL7}, this description of the optimal path $\pi^*$ is
related to the possibility of expressing the hydrodynamic limit for
the process on $\Sigma_N$ whose generator is the adjoint of $L_N$ in
$L_2(d\mu^N_E)$ in terms of \eqref{eq:1} via the nonlocal map $\Phi$.

\subsection*{The asymmetric limit}

Consider the boundary driven asymmetric exclusion process, that is the
process on $\Sigma_N$ with generator given by \eqref{f03} where the
external field $E$ is replaced by $N\alpha$ and the generator is
speeded up by $N$ instead of $N^2$. We consider only the case
$\alpha<0$. According to the previous notation, denote by
$\mu^N_{N\alpha}$ the unique invariant measure of the boundary driven
asymmetric exclusion process with external field $\alpha N$. In the
hydrodynamic scaling limit, it is proved in \cite{ba} that the
empirical density converges to the unique entropy solution to the
inviscid Burgers equation with BLN boundary conditions, namely
\eqref{eq:1} with $E/2$ replaced by $\sinh(\alpha/2)$ and no
viscosity.

Let $\upbar\rho_\asy \in \{\rho_-,\rho_+,1/2\}$ be such that
$\max_{r\in [\rho_-,\rho_+]} \chi(r) = \chi(\upbar\rho_\asy)$. It is
not difficult to check that the stationary profile $\upbar\rho_E$
converges, as $E\to -\infty$, to the constant density profile equal to
$\upbar\rho_\asy$, which is the unique stationary solution to the
inviscid Burgers equation with the prescribed boundary conditions.

By using combinatorial techniques, it is shown in \cite{DLS3} that the
sequence of probability measures $\{\mu^N_{N\alpha}\circ
(\pi^N)^{-1}\}$ on $\mc M$ satisfies a large deviation principle with
speed $N$ and rate function $S_\asy$ defined as follows. Let
\begin{equation}
\label{Fa}
\mc F_\asy := \Big\{ \varphi\in C^1 \big([-1,1]\big)\,:\:
\varphi(\pm 1) = \varphi_\pm \,,\, \varphi' >0 \Big\}\;.  
\end{equation}
Note that $\mc F_E\subset \mc F_\asy$.

Given $\rho\in\mc M$ and $\varphi\in\mc F_\asy$ set
\begin{equation}
\label{Grf-in}
\mc G_{\asy}  (\rho,\varphi)  := \int_{-1}^1\!du \:
\Big\{ \rho\log\rho + (1-\rho)\log(1-\rho)
+(1-\rho)\varphi  -\log\big(1+e^{\varphi}\big) - A_{\asy} \Big\}
\end{equation}
in which the constant $A_{\asy}$ is 
\begin{equation}
\label{cga}
A_{\asy} := \max_{r\in[\rho_-,\rho_+]}  \: \log \chi(r) 
\;=\; \log \chi(\upbar\rho_\asy)  \;.
\end{equation}
Let
\begin{equation}
\label{s-in}
S_{\asy}(\rho) \;:=\; \sup_{\varphi \in \mc F_{\asy}} \; \mc
G_{\asy}(\rho,\varphi) \;.
\end{equation}

The functional $S_{\asy}$ is written in a somewhat different form in
\cite{DLS3}. The above expression is however simply obtained by
replacing the trial function $F$ in \cite{DLS3} by
$e^\varphi/(1+e^\varphi)$. The advantage of the above formulation is
that for each $\rho\in \mc M$ the functional $\mc
G_{\asy}(\rho,\cdot)$ is concave on $\mc F_\asy$.  By choosing
$\varphi = \log[\upbar\rho_\asy/(1-\upbar\rho_\asy)]$ as trial
function in \eqref{s-in} we get a lower bound analogous to
\eqref{diseg}: 
\begin{equation*}
S_{\asy}(\rho) \ge \int_{-1}^{1}\!du\, 
\Big\{\rho\log\frac{\rho}{\upbar\rho_\asy} +
(1-\rho) \log\frac{1-\rho}{1-\upbar\rho_\asy} \Big\}  \;.
\end{equation*}
Note finally that $S_{\asy}$ does not depend on $\alpha<0$. 

We prove in Section \ref{s:asyl} that the functional $S_E$ converges,
as $E\downarrow -\infty$, to $S_\asy$. As discussed in
\cite[Lemma~4.3]{BDGJL8}, the appropriate notion of variational
convergence for rate functionals is the so-called
$\Gamma$-convergence.  Referring e.g.\ to \cite{Braides} for more
details, we just recall its definition.  Let $X$ be a metric space.  A
sequence of functionals $F_n : X \to [0,+\infty]$ is said to
\emph{$\Gamma$-converge} to a functional $F : X \to [0,+\infty]$ if
the following two conditions hold for each $x\in X$. There exists a
sequence $x_n\to x$ such that $\varlimsup_n F_n(x_n) \le F(x)$
(\emph{$\Gamma$-limsup inequality}) and for any sequence $x_n\to x$ we
have $\varliminf_n F_n(x_n) \ge F(x)$ (\emph{$\Gamma$-liminf
  inequality}).

\begin{theorem}
\label{t:gconv}
Let $S_E \,:\mc M \to [0,+\infty]$ be as defined in \eqref{ssg}.  As
$E\downarrow -\infty$, the sequence of functionals $\{S_E\}$
$\Gamma$-converges in $\mc M$ to $S_{\asy}$ defined in \eqref{s-in}.
\end{theorem}

While the above result deals only with the variational convergence of
the quasi-potential, it is reasonable to expect also the convergence
of the dynamical rate functional. More precisely, the dynamic rate
functional \eqref{f02} of the weakly asymmetric exclusion process
should converge, in the appropriate scaling, to the one for the
asymmetric exclusion process. We refer to \cite{BD} for a discussion
of this topic and we mention that the above result has been proven in
\cite{BBMN} for general scalar conservation laws on the real line.

$\Gamma$-convergence implies an upper bound for the infimum over
open sets and a lower bound for the infimum over compacts sets:
For each compact set $\mc K \subset \mc M$ and each open set $\mc
O\subset \mc M$ 
\begin{eqnarray*}
\!\!\!\!\!\!\!\!\!\!\!\! &&
\varliminf_{E\to-\infty}\; \inf_{\rho \in \mc K} S_{E}(\rho)
\;\geq\; \inf_{\rho \in \mc K} S_{\asy}(\rho)\;, \\
\!\!\!\!\!\!\!\!\!\!\!\! && \qquad
\varlimsup_{E\to-\infty}\; \inf_{\rho \in \mc O} S_{E}(\rho)
\;\leq\; \inf_{\rho \in \mc O} S_{\asy}(\rho)\;.
\end{eqnarray*}
The proof of this statement is straightforward and can be found in
\cite[Prop.~1.18]{Braides}. Since $\mc M$ is compact, the previous
fact and Theorems~\ref{t:rfim}, \ref{t:S=V} (iii), \ref{t:gconv}
provide the following asymptotics for the invariant measure $\mu^N_E$.

\begin{corollary}
\label{t:cgconv}
For each closed set $\mc C \subset \mc M$ and each open set $\mc
O\subset \mc M$ 
\begin{eqnarray*}
&&   \varlimsup_{E\to-\infty}\; \varlimsup_{N\to\infty}\; 
\frac 1N \log \mu^N_{E} \big( \pi^N \in \mc C \big)
\;\leq\; - \inf_{\rho \in \mc C} S_{\asy}(\rho)  \\
&&\qquad \varliminf_{E\to-\infty}\;
\varliminf_{N\to\infty} \; \frac 1N \log \mu^N_{E} 
\big( \pi^N\in \mc O\big) 
\;\geq\; -  \inf_{\rho \in \mc O} S_{\asy} (\rho) \; .
\end{eqnarray*}
\end{corollary}

The last topic we discuss is the asymptotic behavior as, $E\to
-\infty$, of the solution to the Euler-Lagrange equation \eqref{Deq}.
More precisely, we show that it converges to the unique maximizer for
\eqref{s-in}. 

Consider the set $\mc F_\asy$ equipped with the topology inherited
from the weak convergence of measures on $[-1,1)$: $\varphi^n\to
\varphi$ in $\mc F_\asy$ if and only if $\int_{-1}^1 \!d\varphi^n \,G
\to \int_{-1}^1 \!d\varphi \,G$ for any function $G$ in
$C_0\big([-1,1)\big)$, the set of continuous functions $G:[-1,1)\to\bb
R$ such that $\lim_{u\uparrow 1} G(u)=0$. The closure of $\mc F_\asy$,
denoted by $\overline {\mc F}_\asy$, consists of all nondecreasing,
c\`adl\`ag functions $\varphi:[-1,1) \to [\varphi_-, \varphi_+]$ such
that $\varphi(-1) = \varphi_-$, $\lim_{u\uparrow 1} \varphi(u) \le
\varphi_+$. By Helly theorem $\overline{\mc F}_\asy$ is a compact
Polish space. Moreover, if $\varphi^n\to\varphi$ in $\overline {\mc
  F_\asy}$ then $\varphi^n(u)\to\varphi(u)$ Lebesgue a.e.

\begin{theorem}
\label{t:cof}
Fix $\rho\in\mc M$. There exists a unique $\phi\in\overline{\mc F}_\asy$ such
that $S_\asy(\rho)=\max_{\varphi\in \mc F_\asy} \mc
G_\asy(\rho,\varphi)= \mc G_\asy(\rho,\phi)$.  Let $\phi_E :=
\Phi(\rho)\in \mc F_E$ be the optimal profile for \eqref{ssg}. As
$E\to -\infty$ the sequence $\{\phi_E\}$ converges to $\phi$ in $\mc
F_\asy$.
\end{theorem}

\section{The nonequilibrium free energy}
\label{s:5}

In this section we analyze the variational problem \eqref{ssg} and
prove items (i) and (ii) in Theorem~\ref{t:S=V}.  We start by proving
an existence and uniqueness result for the Euler-Lagrange equation
\eqref{Deq} together with a $C^1$ dependence of the solution with
respect to $\rho$.  We consider the space $C^1([-1,1])$ endowed with
the norm $\|f\|_{C^1} := \|f\|_\infty + \|f'\|_\infty$ where
$\|g\|_\infty:=\sup_{u\in[-1,1]}|g(u)|$. For each $E< E_0$ the set
$\mc F_E$ defined in \eqref{mcfe} is a convex subset of $C^1([-1,1])$;
we denote by $\upbar{\mc F}_E = \big\{ \varphi \in C^{1}([-1,1]) \,:\:
\varphi(\pm 1) =\varphi_\pm\,,\: \varphi' \ge 0\lor E \big\}$ its
closure in $C^1([-1,1])$.

\begin{theorem}
\label{t:deq} 
Let $E< E_0$. For each $\rho\in\mc M$ there exists in $\mc F_E$ a
unique solution to \eqref{Deq}, denoted by $\Phi(\rho)$.  Furthermore,
\begin{itemize}
\item[{(i)}]{If $\rho\in C([-1,1];[0,1])$ then $\Phi(\rho)\in
    C^2([-1,1])$.}
\item[{(ii)}]{Let $\{\rho^n\}\subset \mc M$ be a sequence converging to
    $\rho$ in $\mc M$. Then $\{\Phi(\rho^n)\}\subset \mc F_E$ converges to 
    $\Phi(\rho)$ in $C^1([-1,1])$.}
\end{itemize}
\end{theorem}

\begin{proof}
The proof is divided in several steps.

\noindent{\sl Existence of solutions.} 
For $E\le 0$, resp.\ $E\in (0,E_0)$, we formulate \eqref{Deq} as an
integral-differential equation informally obtained multiplying
\eqref{Deq} by $\varphi' -E$, resp.\ by $\varphi'$, and integrating
the resulting equation. Existence of solutions will be deduced from
Schauder fixed point theorem.

Given $E< E_0$, $\rho\in \mc M$, and $\varphi\in \upbar{\mc F}_E$,
let
\begin{eqnarray*}
  \mc R^{(1)}(\rho,\varphi;u) & := & 
  \Big[ \rho - \frac{1}{1+e^{\varphi(u)}}\Big] \big[ \varphi'(u) -E
  \big]\;, 
  \\
  \mc R^{(2)}(\rho,\varphi;u) & := & 
  \Big[ \rho - \frac{1}{1+e^{\varphi(u)}}\Big] \varphi'(u) \;.
\end{eqnarray*}
For a fixed $\rho\in\mc M$ and $i=1,2$ we define the
integral-differential operators $\mc K^{(i)}_\rho :\upbar{\mc F}_E \to
C^1([-1,1])$ by
\begin{eqnarray*}
\mc K^{(1)}_\rho (\varphi)\, (u) 
& := &  \varphi_- + \big(\varphi_+-\varphi_-\big)\:
\frac{ 
\displaystyle
\int_{-1}^u \!dv \, 
\exp\left\{ \int_{-1}^v\!dw \, \mc R^{(1)} (\rho,\varphi;w) \right\}
}
{\displaystyle
\int_{-1}^1 \!dv \, 
\exp\left\{ \int_{-1}^v\!dw \, \mc R^{(1)} (\rho,\varphi;w)
\right\}}\;,
\\
&&\phantom{MM}
\\
\mc K^{(2)}_\rho (\varphi)\, (u) 
& := & \varphi_- + E(u+1) \\ 
&&
+ \; 
\big(\varphi_+-\varphi_- -2 E\big) \:
\frac{ 
\displaystyle
\int_{-1}^u \!dv \, 
\exp\left\{ \int_{-1}^v\!dw \, \mc R^{(2)} (\rho,\varphi;w) \right\}
}
{\displaystyle
\int_{-1}^1 \!dv \, 
\exp\left\{ \int_{-1}^v\!dw \, \mc R^{(2)} (\rho,\varphi;w)
\right\}}\;\cdot
\end{eqnarray*}
For $E\leq 0$, resp.\ $E\in (0,E_0)$, we formulate the boundary
problem \eqref{Deq} as a fixed point on $\upbar{\mc F}_E$ for the
operator $\mc K_\rho^{(1)}$, resp.\ $\mc K_\rho^{(2)}$.

Consider first the case $E\le 0$ corresponding to $i=1$.  Simple
computations show that for each $\rho\in\mc M$ the map $\mc
K_\rho^{(1)}$ is a continuous on $\upbar{\mc F}_E$ and $\mc
K_\rho^{(1)}\big(\upbar{\mc F}_E\big)\subset\upbar{\mc F}_E$.  It is
also straightforward to check that there exists a constant
$C_1=C_1(\varphi_-,\varphi_+,E)\in (0,\infty)$ such that for any
$\rho\in\mc M$, $\varphi\in\mc F_E$, and $u$, $v\in[-1,1]$,
\begin{equation}
\label{bk1}
\frac 1 C_1 \;\le\; \frac{d}{du} \mc K_\rho^{(1)}(\varphi) \, (u) 
\; \le \; C_1\;, \quad  
\Big|  \frac{d}{dv} \mc K_\rho^{(1)}(\varphi) \, (v) -
\frac{d}{du} \mc K_\rho^{(1)}(\varphi) \, (u) \Big| 
\; \le\;  C_1 \, |u-v| \;.
\end{equation}
In particular $\mc K_\rho^{(1)}\big(\upbar{\mc F}_E\big)\subset\mc
F_E$.  Notice that $\upbar{\mc F}_E$ is a closed convex subset of
$C^1([-1,1])$ and, by the previous bounds and Ascoli-Arzel\`a theorem,
$\mc{K}_\rho^{(1)}\big(\upbar{\mc F}_E\big)$ has compact closure in
$C^1([-1,1])$.  By Schauder fixed point theorem we get that for each
$\rho\in\mc M$ there exists $\varphi^*\in\upbar{\mc F}_E$ such that
$\mc{K}_\rho^{(1)}(\varphi^*)=\varphi^*$. From \eqref{bk1} it follows
that $\varphi^*\in\mc F_E$ and standard manipulations show that
$\varphi^*$ satisfies \eqref{Deq} Lebesgue a.e.

The case $E\in(0,E_0)$, corresponding to a fixed point for $\mc
K_\rho^{(2)}$, is analyzed in the same way.  In this case, it is
indeed straightforward to check that there exists a constant
$C_2=C_2(\varphi_-,\varphi_+,E)\in (0,\infty)$ such that for any
$\rho\in\mc M$, $\varphi\in\mc F_E$, and $u\in[-1,1]$,
\begin{equation}
\label{bk2}
\frac 1 C_2 \;\le\; \frac{d}{du} \mc K_\rho^{(2)}(\varphi) \, (u) -E 
\;\le\; C_2 \;, \quad
\Big| \frac{d}{dv} \mc K_\rho^{(2)}(\varphi) \, (v) 
- \frac{d}{du} \mc K_\rho^{(2)}(\varphi) \, (u) \Big| 
\;\le C_2\, |u-v| \;.
\end{equation}

\noindent{\sl Uniqueness of solutions.}
Let $\phi\in\mc F_E$, $E\not = 0$, be a solution to \eqref{Deq}; by
chain rule the equation
\begin{equation*}
\Big[ \frac{1}{E} \log \frac{\phi'-E}{\phi'} \Big]'
\;\equiv\;
\frac{\phi''}{\phi'(\phi'-E)} \;=\; \rho -\frac{1}{1+e^\phi}
\end{equation*}
holds Lebesgue a.e. Hence, for each $u\in[-1,1]$,
\begin{equation}
\label{u2}
\frac{1}{E} \log \frac{\phi'(u)-E}{\phi'(u)} 
= \frac{1}{E} \log \frac{\phi'(-1)-E}{\phi'(-1)} 
+\int_{-1}^u\!dv \: 
\Big[ \rho(v) -\frac{1}{1+e^{\phi(v)}}\Big]\;.
\end{equation}

Let $\phi_1,\phi_2\in \mc F_E$ be two solutions to \eqref{Deq}.  If
$\phi_1'(-1)=\phi_2'(-1)$ an application of Gronwall inequality in
\eqref{u2} yields $\phi_1=\phi_2$.  We next assume $\phi_1'(-1)<
\phi_2'(-1)$ and deduce a contradiction.  Recall that $\phi_i'>0\lor
E$ and let $\upbar u := \inf\{ v\in (-1,1] \,:\: \phi_1(v) = \phi_2(v)
\}$ which belongs to $(-1,1]$ because $\phi_1(\pm 1)= \phi_2(\pm 1 )$
and $\phi_1'(-1)<\phi_2'(-1)$.  By definition of $\upbar u$,
$\phi_1(u) < \phi_2(u)$ for any $u\in (-1,\upbar u)$,
$\phi_1(\upbar{u})=\phi_2(\upbar{u})$ and
$\phi_1^\prime(\upbar{u})\geq \phi_2^\prime(\upbar{u})$.  Note that
the real function $(0\lor E, \infty)\ni z\mapsto E^{-1} \log[(z-E)/z]$
is strictly increasing. Therefore from \eqref{u2} we obtain
$\phi_1'(\upbar u) < \phi_2'(\upbar u)$, which is a contradiction and
concludes the proof of the uniqueness.

The case $E=0$ can be treated similarly, with $- ( 1/\phi')'$ in place
of $\{(1/E) \log [ (\phi' - E)/\phi'] \}'$, and was examined in
\cite{bdgjl4}.

\noindent{\sl Claims (i) and (ii).}
Claim (i) follows straightforwardly from the previous analysis. To
prove (ii), let $\phi^n:=\Phi(\rho^n)\in \mc F_E$.  By \eqref{bk1},
\eqref{bk2} and Ascoli-Arzel\`a theorem, the sequence
$\{\phi^n\}\subset \mc F_E$ is precompact in $C^1([-1,1])$.  It
remains to show uniqueness of its limit points.  Consider a
subsequence $n_j$ and assume that $\{\phi^{n_j}\}$ converges to $\psi$
in $C^1([-1,1])$.  Since $\{\rho^{n_j}\}$ converges to $\rho$ in $\mc
M$ and $\{\phi^{n_j}\}$ converges to $\psi$ in $C^1([-1,1])$, for
$E\le 0$, resp.\ for $E\in(0,E_0)$, we have that $\mc
K_{\rho^{n_j}}^{(1)} (\phi^{n_j})$ converges to $\mc K^{(1)}_{\rho}
(\psi)$, resp.\ $\mc K_{\rho^{n_j}}^{(2)} (\phi^{n_j})$ converges to
$\mc K^{(2)}_{\rho} (\psi)$.  In particular, $\psi = \lim_{j}
\phi^{n_j} = \lim_{j} \mc K^{(i)}_{\rho^{n_j}} (\phi^{n_j}) = \mc
K^{(i)}_{\rho} (\psi)$ for $i=1$, $2$. By the uniqueness result,
$\psi=\Phi(\rho)$.  This shows that $\Phi(\rho)$ is the unique
possible limit point of the sequence $\{\phi^n\}$, and concludes the
proof of Claim (ii).
\end{proof}

Fix a path $\rho\equiv \rho_t(u) \in C^{1,0}\big([0,T]\times
[-1,1];[0,1]\big)$ and let $\phi\equiv \Phi(\rho_t) (u)$ be the
solution to \eqref{Deq}. By Theorem \ref{t:deq}, $\phi$ belongs to
$C^{1,2}\big([0,T]\times [-1,1]\big)$.  Note also that, by \eqref{bk1}
and \eqref{bk2}, for each $E< E_0$ there exists a constant $C\in
(0,\infty)$ such that for any $(t,u) \in [0,T]\times[-1,1]$
\begin{equation}
\label{sf'}
\left\{
\begin{array}{l}
{\displaystyle
C^{-1} \;\leq\;  \nabla \phi_t (u) \;\leq\; C
\quad \text{if $E\le 0$},} \\
{\displaystyle
C^{-1} \;\leq\;  \nabla \phi_t (u) -E \;\leq\; C
\quad \text{if $0<E<E_0$}\;.}
\end{array}
\right.
\end{equation}

\begin{lemma}
\label{t:ddeq}
Let $E< E_0$, $T>0$, $\rho\in C^{1,0}\big([0,T]\times
[-1,1];[0,1]\big)$, and $\phi := \Phi(\rho_t)$ be the solution to
\eqref{Deq}.  Then $\phi \in C^{1,2}\big( [0,T]\times [-1,1] \big)$
and $\psi := \partial_t \phi $ is the unique classical solution to the
linear boundary value problem
\begin{equation}
\label{dDeqf}
\begin{cases}
{\displaystyle  
\nabla\Big[ 
\frac{ \nabla \psi_t }{ \nabla \phi_t (\nabla \phi_t -E) }
\Big]
- \frac{e^{\phi_t }}{ \big( 1+ e^{\phi_t} \big)^2} \, \psi_t 
= \partial_t \rho_t
}
& \quad (t,u) \in [0,T]\times (-1,1) \quad\\
\psi_t(\pm 1) = 0 &\quad  t\in [0,T]\;.
\end{cases}
\end{equation}
\end{lemma}

\begin{proof}
Fix $t\in [0,T]$. For $h\neq 0$ such that $t+h\in [0,T]$ define
$\psi^h_t(\cdot)$ by $\psi^h_t(u) := \big[\phi_{t+h}(u) -
\phi_t(u)\big]/h$.  By Theorem \ref{t:deq} (i), $\psi^h_t(\cdot)$
belongs to $C^{2}([-1,1])$. Set $R^h_t := [\rho_{t+h}-\rho_t]/h$; from
\eqref{Deq} it follows that $\psi^h$ solves
\begin{eqnarray}
\label{dDeqfh}
&&\frac{\Delta \psi^h _t}{\nabla \phi_t (\nabla \phi_t-E)}
\; -\; \frac{\Delta \phi _{t+h} 
\big(\nabla \phi_t+\nabla \phi_{t+h} -E \big)}
{\nabla \phi_t (\nabla \phi_t-E) \nabla \phi_{t+h} 
(\nabla \phi_{t+h}-E)} \, \nabla \psi^h_t \\
&&\qquad\qquad\qquad\qquad\qquad\qquad
- \; \frac{e^{\phi_t}}
{\big( 1+ e^{\phi_t} \big) \big( 1+ e^{\phi_{t+h}}\big)} 
\: \frac{e^{h\,\psi^h_t}-1}{h} \;= \; R^h_t
\nonumber 
\end{eqnarray}
for $(t,u) \in [0,T]\times(-1,1)$ with the boundary conditions 
$\psi^h_t(\pm 1) = 0$,  $t\in [0,T]$.

Multiplying the above equation by $\psi^h_t$ and integrating in $du$,
using the inequality $x(e^x-1)\geq 0$ and an integration by parts we
get that
\begin{equation}
\label{pei}
\Big\langle \nabla \psi^h _t, 
\frac{\nabla \psi^h _t}{\nabla \phi _t (\nabla \phi_t- E)}
\Big\rangle \;\leq\;
-\; \langle \psi^h _t, R^h_t\rangle \;+\;
\langle \psi^h _t, F(\phi_t,\phi_{t+h})\nabla \psi^h _t \rangle 
\end{equation}
where 
\begin{eqnarray*}
\!\!\!\!\!\!\!\!\!\!\!\!\! &&
F(\phi_t,\phi_{t+h}) \;:= \; 
\frac{1}{(\nabla \phi_t)^2(\nabla \phi_t-E)^2
\nabla \phi_{t+h}(\nabla \phi_{t+h} -E)} \; \times\\
\!\!\!\!\!\!\!\!\!\!\!\!\! &&
\qquad
\times \; \Big\{ \Delta \phi_t \nabla \phi_{t+h} 
(\nabla \phi_{t+h} -E)(2\nabla \phi_t -E) \\
\!\!\!\!\!\!\!\!\!\!\!\!\! &&
\qquad\qquad\qquad\qquad\qquad\qquad\qquad
-\; \Delta \phi_{t+h} \nabla \phi_t (\nabla \phi_t -E)
(\nabla \phi_{t+h}+ \nabla \phi_t -E) \Big\}\;.
\end{eqnarray*}

For each $t\in[0,T]$,
\begin{equation}
\label{Fto0}
\lim_{h\to 0} \| F(\phi_t,\phi_{t+h}) \|_\infty \;=\; 0\;.
\end{equation}
Indeed, since $\rho\in C^{1,0}\big([0,T]\times [-1,1]\big)$, as $h\to
0$, $\rho_{t+h}(\cdot)\to \rho_t(\cdot)$ in $C([-1,1])$. By
Theorem~\ref{t:deq} (ii), $\phi_{t+h}(\cdot)\to \phi_t(\cdot)$ in
$C^1([-1,1])$. By the differential equation \eqref{Deq},
$\phi_{t+h}(\cdot)\to \phi_t(\cdot)$ in $C^2([-1,1])$. Together with
\eqref{sf'} this concludes the proof of \eqref{Fto0}.

By \eqref{sf'}, Cauchy-Schwarz, and Poincar\'e inequality for the
Dirichlet Laplacian in $[-1,1]$, we obtain from \eqref{pei} that
\begin{eqnarray}
\label{csp}
&& \!\!\!\!\!\!\!\!\!\!\!\!\! 
\frac{1}{C^2} \,
\langle \nabla \psi^h _t, \nabla \psi^h _t \rangle
\;\leq \;
\Big\langle \nabla \psi^h _t, 
\frac{\nabla \psi^h _t}{\nabla \phi_t (\nabla \phi_t - E)}
\Big\rangle \\ 
\nonumber &&
\qquad \leq \; 
\langle\psi^h _t,\psi^h _t \rangle^{1/2} \, \Big[
\langle R^h _t, R^h _t \rangle^{1/2}
\;+\; \| F(\phi_t,\phi_{t+h}) \|_\infty
\langle \nabla \psi^h _t, \nabla \psi^h _t \rangle^{1/2} \Big] \\ 
\nonumber &&
\qquad \leq \; C'\,
\langle\nabla\psi^h _t,\nabla\psi^h _t \rangle^{1/2}\,
\Big[ \langle R^h _t, R^h _t \rangle^{1/2}
+ \| F(\phi_t,\phi_{t+h}) \|_\infty
\langle \nabla \psi^h _t, \nabla \psi^h _t \rangle^{1/2} \Big]
\end{eqnarray}
for some constant $C'>0$.

From \eqref{csp} and \eqref{Fto0} it follows that there exists a
constant $C''>0$ such that
\begin{equation}
\label{ubh1}
\varlimsup_{h\to 0}
\langle \nabla \psi^h _t, \nabla \psi^h _t \rangle
\;\leq\; C'' \, 
\langle \partial_t \rho _t, \partial_t \rho _t \rangle\,,
\quad t\in[0,T]\;.
\end{equation}
Therefore for each $t\in[0,T]$ the sequence $\{\psi^h_t(\cdot)\}$ is
precompact in $C([-1,1])$.  By taking the limit $h\to 0$ in
\eqref{dDeqfh} and using \eqref{Fto0}, it is now easy to show that any
limit point of $\{\psi^h_t(\cdot)\}$ is a weak solution to
\eqref{dDeqf}. By the classical theory on one-dimensional elliptic
problems, see e.g.\ \cite[IV, \S2.1]{M}, there exists a unique weak
solution to \eqref{dDeqf} which is in fact the classical solution
because $\partial_t \rho_t(\cdot)$ belongs to $C([-1,1])$.  This
implies that there exists a unique limit point $\psi_t(\cdot)\in
C^2\big([-1,1]\big)$.  Finally $\psi\in C^{0,2}\big([0,T]
\times[-1,1]\big)$ by the continuous dependence in the $C^2([-1,1])$
topology of the solution to \eqref{dDeqf} w.r.t.\ $\partial_t
\rho_t(\cdot)$ in the $C([-1,1])$ topology.
\end{proof}

We are now in a position to prove two statements of the first main
result of this article.

\begin{proof}[Proof of Theorem~\ref{t:S=V} (i) and (ii)]
We start with Claim (i). The case $E=E_0$ follows from the definition
\eqref{f07} of the functional $S_{E_0}$. Assume $E<E_0$.  By the
convexity of the map $\chi(\rho) = \rho\log\rho +
(1-\rho)\log(1-\rho)$, for each $\varphi\in \mc F_E$ the functional
$\mc G_E (\cdot,\varphi)$ is convex and lower semicontinuous on $\mc
M$.  Hence, by \eqref{ssg}, the functional $S_E$, being the supremum
of convex lower semicontinuous functionals, is a convex lower
semicontinuous functional on $\mc M$.  On the other hand, since the
real function $(0\lor E,\infty) \ni x \mapsto \big[x\log x -
(x-E)\log(x-E)\big]/E$ is strictly concave, Jensen inequality and
$\varphi(\pm 1) = \varphi_\pm$ imply that $\mc{G}_E(\rho,\varphi)$ is
bounded by some constant depending only on $\varphi_\pm$ and $E$.
This proves (i).

Fix $\rho\in\mc{M}$. The strict concavity mentioned above and the
strict concavity of the real function $\bb R\ni x\mapsto
-\log\big(1+e^{x}\big)$ yield that the functional
$\mc{G}_E(\rho,\cdot)$ is strictly concave on $\mc{F}_E$. Thanks to
Theorem~\ref{t:deq}, it easily follows that the supremum on the
r.h.s.\ of \eqref{ssg} is uniquely attained when $\varphi=
\Phi(\rho)$.
\end{proof}

In the proof of the equality between the quasi-potential $V_E$ and the
functional $S_E$, we shall need the following simple observation. 

\begin{lemma}
\label{t:Sden}
For each $\rho\in\mc M$ there exists a sequence $\{\rho^n\}\subset \mc
M$ converging to $\rho$ in $\mc M$ and such that: $\rho^n\in C^2([-1,1])$,
$\rho^n(\pm 1)=\rho_\pm$, $0<\rho^n<1$, $S_E(\rho^n)\to S_E(\rho)$.
\end{lemma}

\begin{proof}
For $E=E_0$, this is obvious from the definition of the functional
$S_{E_0}$. For $E<E_0$, given $\rho\in\mc M$, it is enough to consider
a sequence $\{\rho^n\}\subset C^2\big([-1,1]\big)$ with $\rho^n(\pm
1)=\rho_\pm$ and $0<\rho^n<1$, which converges to $\rho$ $du$ a.e. By
Theorem~\ref{t:S=V} (ii), Theorem~\ref{t:deq} (ii), and dominated
convergence, $S_E(\rho^n)=\mc{G}_E\big(\rho^n,\Phi(\rho^n)\big)
\longrightarrow \mc{G}_E\big(\rho,\Phi(\rho)\big)= S_E(\rho)$.
\end{proof}

\section{The quasi-potential}
\label{s:6}

In this section we characterize the optimal path for the variational
problem \eqref{qp} defining the quasi-potential $V_E$ and conclude the
proof of Theorem~\ref{t:S=V} by showing the equality $V_E=S_E$. The
heuristic argument is quite simple. To the variational problem
\eqref{qp} is associated the following Hamilton-Jacobi equation
\cite{bdgjl2,BDGJL7}. The quasi-potential $V_E$ is the maximal
solution to
\begin{equation}
\label{ham-jac0}
\frac 12\, \Big\langle \nabla \frac{\delta V_E}{\delta\rho} \,,\, 
\chi(\rho) \, \nabla \frac{\delta V_E}{\delta\rho} \Big\rangle
\;+\; \Big\langle\frac{\delta V_E}{\delta\rho} \,,\, \frac 12 
\, \Delta\rho  \;-\; \frac E2 \, \nabla \chi(\rho)  \Big\rangle 
\;=\; 0 
\end{equation}
with the boundary condition that $\delta V_E/ \delta\rho$ vanishes at
the endpoints of $[-1,1]$. Few formal computations show that $S_E$
solves \eqref{ham-jac0}. To check that $S_E$ is the maximal solution
one constructs a suitable path for the variational problem \eqref{qp},
\cite{bdgjl2,BDGJL7}. Since it is not clear how to analyze
\eqref{ham-jac0} directly, we first approximate, as in \cite{bdgjl4},
paths $\pi\in D([0,T];\mc M)$ with $I_T(\pi|\upbar\rho_E)<\infty$ by
smooth paths bounded away from $0$ and $1$ which satisfy the boundary
conditions $\rho_\pm$ at the endpoints of $[-1,1]$. For such smooth
paths we can make sense of \eqref{ham-jac0} and complete the proof.

In the case $E=E_0$, the process is reversible and the picture is well
known. The path which minimizes the variational formula defining the
quasi-potential is the solution of the hydrodynamic equation reversed
in time. The identity between $S_{E_0}$ and $V_{E_0}$ follows easily
from this principle. The proof presented below for $E<E_0$ can be
adapted with several simplifications. It is enough to set $\Phi(\rho)
= \log \{ \upbar\rho_{E_0}/ 1- \upbar\rho_{E_0}\}$ everywhere.

Assume from now on that $E<E_0$.  We first need to recall some
notation introduced in \cite{blm1}. Fix a density profile $\gamma :
[-1,1]\to [0,1]$ and a time $T>0$.  Denote by $\mc F_2 =\mc
F_2(T,\gamma, \rho_\pm)$ the set of trajectories $\pi$ in $C([0,T],
\mc M)$ bounded away from $0$ and $1$ in the sense that for each
$t>0$, there exists $\epsilon >0$ such that $\epsilon \le \pi\le
1-\epsilon$ on $[t,T]$; which satisfy the boundary conditions, $\pi_0
= \gamma$, $\pi_t(\pm 1) = \rho_\pm$, $0\le t\le T$; and for which
there exists $\delta_1$, $\delta_2>0$ such that $\pi_t$ follows the
hydrodynamic equation \eqref{eq:1} in the time interval
$[0,\delta_1]$, $\pi_t$ is constant in the time interval $[\delta_1,
\delta_1 + \delta_2]$ and $\pi_t$ is smooth in time in the time
interval $(\delta_1, T]$.

If the density profile $\gamma$ is the stationary profile
$\upbar\rho_E$, the trajectories $\pi$ in $\mc F_2$ are in fact
constant in the time interval $[0,\delta_1 + \delta_2]$. Since they
are also smooth in time in $(\delta_1, T]$, we deduce that they are
smooth in time in the all interval $[0,T]$. Moreover, since
$\upbar\rho_E$ is bounded away from $0$ and $1$, there exists
$\epsilon >0$ such that $\epsilon \le \pi\le 1-\epsilon$ on $[0,T]$. 

Assume that $\gamma = \upbar\rho_E$ and recall from the proof of
\cite[Theorem 4.6]{blm1} the definition of the sequence of
trajectories $\{\pi_\epsilon : \epsilon >0\}$. Since a path $\pi$ in
$\mc F_2$ is in fact constant in the time interval $[0,b]$, each 
$\pi_\epsilon$ is smooth in space and time. In particular, let
\begin{equation}
\label{dDd}
\mc D_0 \;:=\; C^{\infty,\infty} \big([0,T]\times[-1,1] \big) \cap 
\mc F_2 \;. 
\end{equation}
Theorem 4.6 in \cite{blm1} can be rephrased in the present context as

\begin{theorem}
\label{s01}
For each $\pi$ in $D([0,T], \mc M)$ such that
$I_T(\pi|\upbar\rho_E)<\infty$, there exists a sequence
$\{\pi^n\}\subset \mc D_0$ converging to $\pi$ in $D\big([0,T];\mc
M\big)$ such that $I_T(\pi^n|\upbar\rho_E)$ converges to
$I_T(\pi|\upbar\rho_E)$.
\end{theorem}

The first two lemmata of this section state that, for smooth paths,
the functional $S_E$ satisfies \eqref{ham-jac0}.  Recall that for
$\rho\in\mc M$ we denote by $\Phi(\rho)\in\mc F_E$ the unique solution
to \eqref{Deq}.

\begin{lemma}
\label{dotS}
Let $E< E_0$, $T>0$, $\pi\in \mc D_0$, and $\Gamma: [0,T]\times[-1,1]
\to \bb R$ be defined by
\begin{equation}
\label{dGt}
\Gamma_t \; := \; \log \frac{\pi_t}{1-\pi_t} \; -\; \Phi(\pi_t) \; .
\end{equation}
Then
\begin{equation}
\label{nome2}
S_E(\pi_T) - S_E(\pi_0) 
\; =\;  \int_0^T  \!dt\, 
\langle \Gamma_t, \partial_t \pi_t \rangle \; .
\end{equation}
\end{lemma}

\begin{proof}
Let $\phi \equiv \phi_t (u):= \Phi(\pi_t) \,(u)$, $(t,u)\in
[0,T]\times[-1,1]$. By Lemma \ref{t:ddeq}, $\phi$ belongs to
$C^{1,2}\big([0,T]\times[-1,1]\big)$.  Since $\phi_t(\pm 1) =
\varphi_\pm$, then $\partial_t \phi_t(\pm 1)=0$, $t\in[0,T]$.  By
Theorem~\ref{t:S=V} (ii), dominated convergence, an explicit
computation, and an integration by parts,
\begin{eqnarray*}
\frac {d}{dt} S_E(\pi_t )
&=& \frac {d}{dt} \: \mc{G}_E\big(\pi_t \,,\, \Phi(\pi_t) \big) \\
&=& \big\langle \Gamma_t \,,\, \partial_t \pi_t\big\rangle
\;+\; \Big\langle \partial_t \phi_t \,,\, 
\frac{\Delta \phi_t}{\nabla\phi_t (\nabla\phi_t -E)} 
\;+\; \frac 1{1+ e^{\phi_t}} \;-\; \pi_t \Big\rangle\;.
\end{eqnarray*}
The lemma follows noticing that the last term vanishes by \eqref{Deq}.
\end{proof}

Let
\begin{equation}
\label{dMd}
\mc{M}_0 \;:=\; \big\{ \rho \in C^2\big([-1,1]\big)\, :\:
\rho(\pm 1) =\rho_\pm \,,\; 0 < \rho < 1 \big\}\;.
\end{equation}

\begin{lemma}
\label{ham-jac}
Let $E< E_0$, $\rho\in \mc M_0$, and $\Gamma: [-1,1] \to \bb R$ be
defined by
\begin{equation}
\label{h.0}
\Gamma  \;: = \; \log \frac{\rho}{1-\rho} \;-\; \Phi(\rho)\;.
\end{equation}
Then,
\begin{equation}
\label{h.1}
\big\langle \nabla \Gamma \,,\, \chi(\rho) \, \nabla \Gamma \big\rangle
\; - \; \big\langle \nabla \rho \;-\; E \chi(\rho) 
\,,\, \nabla \Gamma \big\rangle \; =\;  0 \; . 
\end{equation}
\end{lemma}

\begin{proof}
As before we let $\phi \equiv \phi (u) := \Phi(\rho) \,(u)$, $u\in
[-1,1]$.  By Theorem~\ref{t:deq} (i), $\phi$ belongs to $C^2([-1,1])$.
By the definition of $\Gamma$ in \eqref{h.0}, statement \eqref{h.1} is
equivalent to
\begin{equation*}
\big\langle \nabla \rho \,,\, - \nabla \phi + E \big\rangle 
\;+\; \big\langle - \nabla \phi \,,\, \chi(\rho) 
\,(- \nabla \phi + E) \big\rangle \;=\; 0\; .
\end{equation*}
The above equation holds if and only if 
\begin{equation*}
\Big\langle \nabla \Big( \rho - \frac{e^\phi}{1 + e^\phi} \Big) \,,\, 
\nabla \phi - E \Big\rangle 
\;+\; \Big\langle \nabla \Big( \frac{e^\phi}{1 + e^\phi} \Big)\,,\, 
\nabla \phi - E \Big\rangle  
\;-\; \big\langle \nabla \phi \,,\, \chi(\rho) (\nabla \phi - E)
\big\rangle \;=\; 0 \;.
\end{equation*}
Since $e^{\phi(\pm 1)}/[1 + e^{\phi(\pm 1)}] 
= e^{\varphi_\pm} /[1 + e^{\varphi_\pm}] = \rho_\pm=\rho(\pm 1)$,
integrating by parts the previous equation, it becomes
\begin{equation*}
\Big\langle \rho - \frac{e^\phi}{1 + e^\phi}\,,\,  \Delta \phi \Big\rangle 
\;-\; \Big\langle \Big( \frac{e^\phi}{\big(1 + e^\phi\big)^2} 
\;-\; \chi(\rho) \Big) \nabla \phi \,,\, \nabla \phi - E \Big\rangle 
\;=\;0 \;.
\end{equation*}

At this point the explicit expression for $\chi$ given by 
$\chi(\rho)=\rho(1-\rho)$ plays a crucial role. Indeed, for such $\chi$,
\begin{equation*}
\frac{e^\phi}{(1 + e^\phi)^2} - \chi(\rho) \;=\;
- \, \Big( \frac{e^\phi}{1 + e^\phi} - \rho \Big)\:
\Big( \frac{e^\phi}{1 + e^\phi} - (1 - \rho) \Big)\;,
\end{equation*}
so that \eqref{h.1} is equivalent to
\begin{equation*}
\Big\langle \rho - \frac{e^\phi}{1 + e^\phi}\,,\,  \Delta \phi
\;+\; \nabla \phi \,(\nabla \phi - E)\,
\Big( 1 - \rho - \frac{e^\phi}{1 + e^\phi} \Big) \Big\rangle 
\;=\; 0\;, 
\end{equation*}
which holds true because $\phi=\Phi(\rho)$ solves \eqref{Deq}.
\end{proof}

We next prove the first half of the equality $V_E=S_E$. In fact the
argument basically shows that any solution to the Hamilton-Jacobi
equation \eqref{ham-jac0} gives a lower bound on the quasi-potential.

\begin{proof}[Proof of Theorem~\ref{t:S=V}: the inequality $V_E \geq S_E$]
In view of the variational definition of $V_E$ in \eqref{qp}, to prove
the lemma we need to show that for each $\rho\in\mc M$ we have
$S_E(\rho) \leq I_T(\pi |\upbar\rho_E)$ for any $T>0$ and any path
$\pi\in D\big([0,T];\mc{M}\big)$ such that $\pi_T=\rho$. 
  
Assume firstly that $\rho\in\mc M_0$ and consider only paths
$\pi\in\mc D_0$. Of course the energy $\mc Q(\pi)$ of such a path
$\pi$ is finite. In view of the variational definition of $I_T(\pi
|\upbar\rho_E)$ given in \eqref{f01}, \eqref{f02}, to prove that
$S_E(\rho) \leq I_T(\pi |\upbar\rho_E)$ it is enough to exhibit some
function $H \in C^{1,2}_0 ([0,T] \times [-1,1])$ for which $S_E(\rho)
\leq \hat J_{T,H,\upbar{\rho}_E} (\pi)$.  We claim that $\Gamma$ given
in \eqref{dGt} fulfills these conditions.  Let $\phi \equiv \phi_t(u)
:= \Phi(\pi_t)\,(u)$.  Since $\pi \in \mc D_{0}$, by
Lemma~\ref{t:ddeq} $\Gamma\in C^{1,2} ([0,T] \times [-1,1])$. On the
other hand, since $\pi_t(\pm 1)=\rho_\pm$ and $\phi_t(\pm
1)=\varphi_\pm$, $\Gamma_t (\pm 1) =0$, $t\in[0,T]$; whence $\Gamma\in
C^{1,2}_0 ([0,T] \times [-1,1])$.  Recalling the definition of the
functional $\hat J_{T,\Gamma,\upbar{\rho}_E}$, after an integration by
parts, we obtain that
\begin{equation*}
\hat J_{T,\Gamma,\upbar{\rho}_E} (\pi) \;=\;
\int_0^T \!dt\, \Big[ \big\langle \Gamma_t,\partial_t \pi_t \big\rangle 
+ \frac 12 \big\langle  \nabla \Gamma_t, \nabla \pi_t 
- E \chi(\pi_t) \big\rangle 
- \frac 12 \big\langle  \chi(\pi_t) , (\nabla \Gamma_t)^2\big\rangle  
\Big] \; .
\end{equation*}
From Lemmata \ref{dotS}, \ref{ham-jac} and since $\pi_0 =
\upbar\rho_E$, $S_E(\upbar\rho_E) =0$, it follows that $\hat
J_{T,\Gamma,\upbar{\rho}_E} (\pi) = S_E(\rho)$, which proves the
statement for $\rho\in\mc M_0$ and paths $\pi\in \mc D_0$. 

Let now $\rho\in\mc M$ and consider an arbitrary path $\pi\in
D\big([0,T];\mc M\big)$ such that $\pi_T=\rho$. With no loss of
generality we can assume $I_T(\pi|\upbar\rho_E)<\infty$. Let
$\{\pi^n\}\subset \mc D_0$ be the sequence given by Theorem
\ref{s01}. The result for $\rho\in \mc M_0$ and paths in $\mc D_0$,
together with the lower semicontinuity of $S_E$, yield
\begin{equation*}
I_T(\pi|\upbar\rho_E) = 
\lim_{n\to\infty} I_T(\pi^n|\upbar\rho_E) \ge  
\varliminf_{n\to\infty}  S_E(\pi^n_T) \ge S_E (\pi_T) 
=  S_E(\rho)\;,
\end{equation*}
which concludes the proof.
\end{proof}

To prove the converse inequality $V_E\le S_E$ on $\mc M$, we need to
characterize the optimal path for the variational problem \eqref{qp}.
The following lemma explains which is the right candidate.  

Denote by $C^\infty_K (\Omega_T)$ the smooth functions $H : \Omega_T
\to \bb R$ with compact support.  For a trajectory $\pi$ in $D([0,T],
\mc M)$, let $\mc H^{1}_0 (\chi(\pi))$ be the Hilbert space induced by
$C^\infty_K (\Omega_T)$ endowed with the scalar product defined by
\begin{equation*}
\<\!\< G,H \>\!\>_{1,\chi(\pi)} \;=\; 
\int_0^T dt\, \int_{-1}^1 du\, 
(\nabla G) (t,u) \, (\nabla H)(t,u)  \, \chi(\pi(t,u))  \;.
\end{equation*}
Induced means that we first declare two functions $F$, $G$ in
$C^{\infty}_K(\Omega_T)$ to be equivalent if $\<\!\< F-G, F-G
\>\!\>_{1,\chi(\pi)} =0$ and then we complete the quotient space with
respect to scalar product.  Denote by $\Vert \cdot \Vert_{1,
  \chi(\pi)}$ the norm associated to the scalar product $\<\!\<\cdot,
\cdot \>\!\>_{1,\chi(\pi)}$.

Repeating the arguments of the proof of Lemma 4.7 in \cite{blm1}, we
obtain an explicit expression of the rate function $I_T(\pi|\gamma )$
in terms of a solution to an elliptic equation.

\begin{lemma}
\label{g07}
Fix a trajectory $\pi$ in $\mc D_0$. For each $0\le t\le T$, let $H_t$
be the unique solution to the elliptic equation
\begin{equation}
\label{f06}
\left\{
\begin{array}{l}
\partial_t \pi_t  \;=\; (1/2)
\Delta \pi_t \;-\; \nabla \big\{ \chi(\pi_t) 
\big[ (E/2) + \nabla H_t \big] \big\}\;, \\
H_t(\pm 1) \;=\; 0\;.
\end{array}
\right.
\end{equation}
Then, $H$ is smooth on $[0,T] \times [-1,1]$ and
\begin{equation}
\label{f13}
I_T(\pi|\pi_0 ) \;=\; \frac 12  \Vert H \Vert^2_{1, \chi(\pi)}\;.
\end{equation}
\end{lemma}

We could have used next lemma to prove the inequality $V_E\ge S_E$; we
presented the separate argument before for its simplicity. On the
other hand, \eqref{I=} clearly suggests that the optimal path for the
variational problem \eqref{qp} is obtained by taking a path which
satisfies \eqref{f07bis} with $K=0$. Recall that $\Phi(\rho)$ denotes
the solution to \eqref{Deq}.

\begin{lemma}
\label{I=I*+S}
Let $E< E_0$, $T>0$, $\gamma \in \mc M_0$, and $\pi \in \mc D_{0}$
such that $I_T(\pi|\gamma)<\infty$.  Then, there exists $K$ in
$C^{1,2}_0\big([0,T]\times [-1,1]\big)$ such that $\pi$ is a classical
solution to
\begin{equation} 
\label{f07bis}
\begin{cases}
{\displaystyle
\partial_t \pi_t 
+ \frac E2 \,  \nabla \chi(\pi_t) 
= - \frac 12  \, \Delta \pi_t 
+  \nabla \big[ \chi(\pi_t) \nabla  \big( \Phi(\pi_t) + K_t \big) 
\big]  
}
\\
{\displaystyle
\vphantom{\Big\{}
\pi_t(\pm 1) = \rho_\pm
}
\\
{\displaystyle
\pi_0 = \gamma\;.
}
\end{cases}
\end{equation}
Furthermore, 
\begin{equation}
\label{I=}
I_T(\pi|\gamma)=     S_E(\pi_T) -  S_E(\gamma) + 
\frac 12 \, \big\| K \big\|_{1,\chi(\pi)}^2 \;\cdot
\end{equation}
\end{lemma}

\begin{proof} 
Note that $\gamma = \pi_0$ because we assume the rate function to be
finite. Denote by $H$ the smooth function introduced in Lemma
\ref{g07} and let $\Gamma$ be as defined in \eqref{dGt}. We claim that
$K:=\Gamma-H$ meets the requirements in the lemma. As before we have
that $\Gamma$ belongs to $C^{1,2}_0\big([0,T]\times [-1,1]\big)$.
Hence, $K$ also belongs to this space because $H$ is smooth and
vanishes at the boundary of $[-1,1]$. The equation \eqref{f07bis}
follows easily from \eqref{f06} replacing $H$ by $\Gamma - K$. To
prove identity \eqref{I=}, consider \eqref{nome2} and express
$\partial_t \pi_t$ in terms of the differential equation in
\eqref{f06}. Since $H= \Gamma - K$, after an integration by parts we
get that $S_E(\pi_T) - S_E(\gamma)$ is equal to
\begin{equation*}
- \frac 12 \int_0^T \!dt\, 
\big\langle \nabla \Gamma_t,  \nabla \pi_t  - E\,  \chi(\pi_t)\big\rangle 
\;+\; \int_0^T \!dt\, \big\langle \nabla \Gamma_t\,,\, \chi(\pi_t) 
\nabla \big( \Gamma_t -K_t \big) \big\rangle\; .
\end{equation*}
By Lemma \ref{ham-jac}, the previous expression is equal to
\begin{equation*}
\frac 12  \int_0^T \!dt \, \big\langle \nabla \Gamma_t\,,\, 
\chi(\pi_t) \nabla \Gamma_t \big\rangle
\;-\; \int_0^T \!dt \, \big\langle \nabla \Gamma_t \,,\, 
\chi(\pi_t) \nabla K_t \big\rangle\;.
\end{equation*}
Since $K=\Gamma-H$, we finally get that 
\begin{equation*}
S_E(\pi_T) \;-\; S_E(\gamma) \;+\; \frac 12 \, 
\big\| K \big\|_{1,\chi(\pi)}^2 \;=\;
\frac 12 \, \big\| H \big\|_{1,\chi(\pi)}^2\;,
\end{equation*}
which, in view of \eqref{f13}, concludes the proof.
\end{proof}

We next show how a solution to the (nonlocal) equation \eqref{f07bis}
with $K=0$ can be obtained by the algorithm presented below the
statement of Theorem~\ref{t:S=V}. Recall that such algorithm requires
to solve \eqref{Deq} only for the initial datum and then to solve the
(local) hydrodynamic equation \eqref{eq:1}. Note indeed that by
setting $\pi^*_t:=\rho^*_{-t}$, where $\rho^*$ is defined in the next
lemma, then $\pi^*$ solves the differential equation in \eqref{f07bis}
with $K=0$.

Fix $E< E_0$, $\gamma\in\mc M_0$ and set $G := e^{\Phi(\gamma)}/
[1+e^{\Phi(\gamma)}]$. By Theorem~\ref{t:deq} the profile $G$ belongs
to $C^4([-1,1])$, it is strictly increasing and satisfies $G(\pm
1)=\rho_\pm$. Denote by $F\equiv F_t(u)\in C^{1,4}\big([0,\infty)
\times [-1,1]\big)$ the solution to the hydrodynamic equation
\eqref{eq:1} with $\gamma$ replaced by $G$. By the maximum principle,
$\rho_- \le F \le \rho_+$.

\begin{lemma}
\label{t:cadh}
Let $\psi :=\log [ F / (1-F) ]$. Then, $\psi$ belongs to
$C^{1,4}\big([0,\infty) \times [-1,1]\big)$ and satisfies $\nabla \psi
> 0\lor E$.  Let $\rho^*\equiv \rho^*_t(u)$ be defined by
\begin{equation}
\label{r*}
\rho^* := \frac{1}{1+ e^{\psi}} +
\frac{\Delta \psi}{\nabla\psi (\nabla \psi -E)}\;\cdot 
\end{equation} 
Then, $\rho^*$ belongs to $C^{1,2}\big([0,\infty) \times [-1,1]\big)$,
satisfies $\rho^*_t(\pm 1)=\rho_\pm$, $0<\rho^*<1$, and solves
\begin{equation}
\label{adhr*}
\begin{cases}
{\displaystyle
\partial_t \rho^*_t - \frac E2 \,  \nabla \chi(\rho^*_t)
= \frac 12  \, \Delta \rho^*_t - 
\nabla \big[ \chi(\rho^*_t) \nabla \Phi(\rho^*_t) \big]} \\
{\displaystyle \vphantom{\Big\{}
\rho^*_t(\pm 1) = \rho_\pm} \\
{\displaystyle
\rho^*_0 = \gamma\;.}
\end{cases}
\end{equation}
\end{lemma}

\begin{proof}
Let $\psi: [0,\infty)\times[-1,1] \to \bb R$ be given by $\psi = \log
\{ F/1-F\}$ and set
\begin{equation*}
\tau := \sup \big\{ t\ge 0 \,:\: \nabla \psi_s (u) > 0 \lor E 
\textrm{ for all } (s,u) \in [0,t]\times [-1,1] \big\} \;.
\end{equation*}
Since $\nabla\psi_0 = \nabla \Phi(\gamma) > 0\lor E$, $\tau>0$ by
continuity. We show at the end of the proof that $\tau=\infty$.

A straightforward computation shows that $\psi$ solves
\begin{equation}
\label{epsi}
\begin{cases}
{\displaystyle 
\partial_t \psi = \frac 12 \, \Delta \psi +\frac{1}{2} \,
\frac{1-e^{\psi}}{1+e^{\psi}} \nabla \psi \big( \nabla \psi -
E \big) } \\
\vphantom{\Big\{} \psi_t(\pm 1 ) = \varphi_\pm \\
\psi_0 = \Phi(\gamma)\;.
\end{cases}
\end{equation}
Since $\psi\in C^{1,4}\big([0,\infty) \times [-1,1]\big)$, definition
\eqref{r*} yields $\rho^*\in C^{1,2}\big([0,\tau)\times [-1,1]\big)$
and $\rho^*_0=\gamma$. On the other hand, from \eqref{epsi} we deduce
that for any $t\in[0,\tau)$
\begin{equation*}
\Delta \psi_t(\pm 1) 
+ \frac{1-e^{\varphi_\pm}}{1+e^{\varphi_\pm}} 
\nabla \psi_t(\pm 1) \big[ \nabla \psi_t(\pm 1) - E \big] =0 \;.
\end{equation*}
Whence, again by \eqref{r*},
\begin{equation*}
\rho^*_t (\pm 1) = \frac{1}{1+ e^{\varphi_\pm}} -
\frac{1-e^{\varphi_\pm}}{1+e^{\varphi_\pm}} = \rho_\pm\;.
\end{equation*}
By using \eqref{epsi}, a long and tedious computation that we omit
shows that $\rho^*_t$, $t\in [0,\tau)$, solves the differential
equation in \eqref{adhr*}.

We next show that $0 < \rho^* < 1$.  Since $\gamma\in \mc M_0$, there
exists $\delta\in (0,1)$ such that $\delta \le \gamma \le 1-\delta$.
We claim that $\min\{\rho_-,1-\rho_+,\delta\} \le \rho^* \le
\max\{\rho_+,1-\rho_-,1-\delta\}$. Fix $t\in (0,\tau)$ and assume that
$\rho^*_t(\cdot)$ has a local maximum at $u_0\in (-1,1)$. Since
$\rho^*$ solves \eqref{adhr*}, since $\Phi(\rho^*)$ solves \eqref{Deq}
and since $\nabla \rho^*_t(u_0) =0$, $\Delta \rho^*_t(u_0) \le 0$,
\begin{eqnarray*}
&&\!\!\!\! \partial_t \rho^*_t(u_0) \;=\; \frac 12 \Delta\rho^*_t(u_0)
\;-\;\chi(\rho^*_t(u_0)) \, \Delta \Phi(\rho^*_t) (u_0) \\
&&\quad \le  -\chi(\rho^*_t(u_0)) \,  \nabla \Phi(\rho^*_t) (u_0) 
\, \big[ \nabla \Phi(\rho^*_t) (u_0) -E \big] 
\Big[ \rho^*_t(u_0) - \frac 1{ 1+ e^{\Phi(\rho^*_t) (u_0)}} \Big]\;.
\end{eqnarray*}
Assume now that $\rho^*_t(u_0)> 1-\rho_-$.  Since $\Phi(\rho^*_t) \ge
\varphi_-$ we deduce $\rho^*_t(u_0) - [1+e^{\Phi(\rho^*_t)
  (u_0)}]^{-1} > 1-\rho_- - [1+e^{\varphi_-}]^{-1}=0$. As $\nabla
\Phi(\rho^*_t) (u_0) > 0\lor E$, we get $\partial_t \rho^*_t(u_0) <
0$. In particular, by a standard argument, $\rho^* \le
\max\{\rho_+,1-\rho_-,1-\delta\}$. The proof of the lower bound is
analogous.

We conclude the proof showing that $\tau=\infty$. Assume that
$\tau<\infty$. Since for each $t\in [0,\tau)$, $\rho^*_t$ belongs to
$\mc M_0$, it follows from \eqref{r*} that $\Phi(\rho^*_t) = \psi_t$,
$t\in [0,\tau)$. By Theorem \ref{t:deq} (ii), $\Phi(\rho^*_\tau) =
\psi_\tau$ so that $\nabla \psi_\tau = \nabla \Phi (\rho^*_\tau) >
E\vee 0$ because $\Phi (\rho^*_\tau)$ belongs to $\mc F_E$. By
continuity, there exists $\delta >0$ such that $\nabla \psi_t > E\vee
0$ for $\tau\le t < \tau +\delta$. This contradicts the definition of
$\tau$.
\end{proof}

Fix a density profile $\gamma:[-1,1]\to [0,1]$, a time $T>0$ and
consider the solution $\rho^*$ to \eqref{adhr*}. Let $\lambda_t(\cdot)
= \rho^*_{T-t}(\cdot)$. Clearly, $\lambda$ is the solution to
\eqref{f07bis} in the time interval $[0,T]$ with $K=0$ and initial
condition $\lambda_0 = \rho^*_T$. In particular, by \eqref{I=}, 
\begin{equation*}
I_T(\lambda|\rho^*_T) \;=\;  S_E(\gamma) \;-\;  S_E(\rho^*_T) \;.
\end{equation*}
In the next lemma we prove that $\rho^*_T$ converges to $\upbar\rho_E$
as $T\to\infty$. Letting $T\uparrow\infty$ in the previous formula, we
see that the time reversed trajectory of \eqref{adhr*} is the natural
candidate to solve the variational formula defining the
quasi-potential. This argument is made rigorous in the next
paragraphs.

By standard properties of parabolic equations on a bounded interval,
see e.g.\ \cite{GK}, as $t\to \infty$, the solution to \eqref{eq:1}
converges, in a strong topology, to the unique stationary solution
$\upbar\rho_E$. Such convergence implies that the path $\rho^*$, as
defined in Lemma~\ref{t:cadh}, also converges to $\upbar\rho_E$ as
$t\to\infty$. This is the content of the next lemma. This result will
permit to use the time reversal of $\rho^*$ as a trial path in the
variational problem \eqref{qp}.

\begin{lemma}
\label{t:conv}
Let $E< E_0$, $\gamma\in\mc M_0$, and $\rho^*$ be defined as in
Lemma~\ref{t:cadh}. As $t\to\infty$, the profile $\rho^*_t\in \mc M_0$
converges to $\upbar\rho_E$ in the $C^1([-1,1])$ topology, uniformly
for $\gamma\in \mc M_0$.
\end{lemma}

\begin{proof}
Recall the notation introduced just before Theorem \ref{t:cadh}. Let
$\rho$ be the solution to \eqref{eq:1}.  In \cite[Theorem~4.9]{GK} it
is shown that, as $t\to\infty$, the profile $\rho_t$ converges to
$\upbar\rho_E$ in the $C^1([-1,1])$ topology, uniformly for $\gamma\in
\mc M_0$. By the methods there developed, it is however
straightforward to prove this statement in the $C^3([-1,1])$ topology.
In particular, $F_t$ converges to $\upbar\rho_E$ in the $C^3([-1,1])$
topology to $\upbar\rho_E$ so that $\psi_t$ converges to
$\log[\upbar\rho_E/(1+\upbar\rho_E)]=\upbar\varphi_E$ in the
$C^3([-1,1])$ topology uniformly in $\gamma\in \mc M_0$. Since
$\Phi(\upbar\rho_E)=\upbar\varphi_E$, the statement now follows from
\eqref{r*}.
\end{proof}

We next show that profiles close to $\upbar\rho_E$ in a strong
topology can be reached with a small cost. 

\begin{lemma}
\label{t:join}
Let $E< E_0$ and $\delta\in (0,1)$. Then, there exist $T>0$ and
constant $C=C(E,\rho_\pm,\delta)>0$ such that the following hold.  For
each $\rho\in C^1([-1,1])$ satisfying $\rho(\pm 1)=\rho_\pm$ and
$\delta \le \rho \le 1-\delta$, there exists a path $\hat\pi\in
D\big([0,T];\mc M\big)$ such that $\hat\pi_T=\rho$ and
\begin{equation*}
I_T(\hat\pi|\upbar\rho_E) \le C \, 
\big\|\rho -\upbar\rho_E\big\|_{C^1}^2\;.
\end{equation*}
\end{lemma}

\begin{proof}
Simple computations show that $T=1$ and the straight path $\hat\pi_t =
\upbar\rho_E +t (\rho-\upbar\rho_E)$ meet the requirements. For $E=0$,
in \cite[Lemma~5.7]{bdgjl4} a more clever path is chosen which yields
a bound in terms of the $L_2$ norm of $\rho-\upbar\rho_E$.
\end{proof}

We can now conclude the proof of Theorem~\ref{t:S=V}.

\begin{proof}[Proof of Theorem~\ref{t:S=V}: the inequality $V_E \leq S_E$]
Given $\rho\in \mc M$ and $\delta>0$ we need to find $T>0$ and a path
$\pi^*\in D\big([0,T];\mc M\big)$ such that $\pi_T^*=\rho$ and
$I_T(\pi^*|\upbar\rho_E) \le S_E(\rho) +\delta$. By
Lemma~\ref{t:Sden}, there exists a sequence $\{\rho^n\}\subset\mc M_0$
converging to $\rho$ in $\mc M$ and such that $S_E(\rho^n) \to
S_E(\rho)$. Let $\rho^{*,n}$ be the path constructed in
Lemma~\ref{t:cadh} with $\gamma$ replaced by $\rho^n$ and pick
$\epsilon>0$ to be chosen later. By Lemma~\ref{t:conv}, there exists a
time $T_1=T_1(\epsilon)>0$ independent of $n$ such that
$\|\rho^{*,n}_{T_1} -\upbar\rho_E\|_{C^1} \le \epsilon$. Whence, by
Lemma~\ref{t:join}, there exists a time $T_2>0$, still independent of
$n$, and a path $\hat\pi^n_t$, $t\in[0,T_2]$ such that $\hat\pi^n_{0}
=\upbar\rho_E$, $\hat\pi^n_{T_2}=\rho^{*,n}_{T_1}$ and
$I_{T_2}(\hat\pi^n|\upbar\rho_E) \le \beta_\epsilon$, where
$\beta_\epsilon$ vanishes as $\epsilon \to 0$ and is independent of
$n$. We now set $T:=T_1+T_2$ and let $\pi^{*,n}_t$, $t\in[0,T]$ be the
path defined by
\begin{equation*}
\pi^{*,n}_t:=
\begin{cases}
\hat\pi^n_t & t\in [0,T_2] \\
\rho^{*,n}_{T-t} & t\in (T_2,T] 
\end{cases}
\end{equation*}
which satisfies $\pi^{*,n}_0=\upbar\rho_E$ and $\pi^{*,n}_T=\rho^n$.
The covariance of $I$ w.r.t.\ time shifts, Lemmata~\ref{I=I*+S} and
\ref{t:cadh} yield
\begin{eqnarray}
\label{ubn}
\nonumber
I_T(\pi^{*,n}|\upbar\rho_E) 
& = & I_{T_2}(\hat\pi^{n}|\upbar\rho_E) 
+ I_{T_1}(\rho^{*,n}_{T_1-\cdot}|\rho^{*,n}_{T_1} ) \\
&\le& \beta_\epsilon + S_E(\rho^n) - S_E(\rho^{*,n}_{T_1})
\; \le\;  \beta_\epsilon + S_E(\rho^n) \;.
\end{eqnarray}
Since $S_E(\rho^n)\to S_E(\rho)< \infty$ and $I_T(\cdot|\upbar\rho_E)$
has compact level sets, see Theorem~\ref{s02}, the bound \eqref{ubn}
implies precompactness of the sequence $\{\pi^{*,n}\}\subset
D\big([0,T];\mc M\big)$. It therefore exists a path $\pi^{*}$ and a
subsequence $n_j$ such that $\pi^{*,n_j}\to \pi^{*}$ in
$D\big([0,T];\mc M\big)$. In particular $\pi^*_T =\lim_{j}
\pi^{*,n_j}_T=\lim_{j} \rho^{n_j}=\rho$. The lower semicontinuity of
$I_T(\cdot|\upbar\rho_E)$ and \eqref{ubn} now yield
\begin{equation*}
I_T(\pi^{*}|\upbar\rho_E) 
\;\le\; \varliminf_{j\to\infty} I_T(\pi^{*,n_j}|\upbar\rho_E) 
\;\le\; \beta_\epsilon + \lim_{j\to\infty} S_E(\rho^{n_j})
\;=\; \beta_\epsilon + S_E(\rho)
\end{equation*}
which, by choosing $\epsilon$ so that $\beta_\epsilon \le \delta$,
concludes the proof.
\end{proof}

\section{The asymmetric limit}
\label{s:asyl}

In this section we discuss the asymmetric limit $E\to -\infty$ and
prove Theorems~\ref{t:gconv} and \ref{t:cof}. 

\begin{proof}[Proof of Theorem~\ref{t:gconv}: $\Gamma$-liminf
  inequality]
Fix $\rho\in\mc M$ and a sequence $\{\rho_E\}\subset \mc M$ converging
to $\rho$ in $\mc M$ as $E\to -\infty$. We need to prove that
$\varliminf_E S_{E}(\rho_E) \ge S_{\asy}(\rho)$.

Let $J_E$ be such that \eqref{css} holds; it is straightforward to
check that
\begin{equation*}
\lim_{E\to -\infty} \frac{J_E}{E} = \max_{r\in[\rho_-,\rho_+]}
\chi(r)\;, 
\end{equation*}
whence, recalling that $A_E$ has been defined in \eqref{cge} and
$A_\asy$ in \eqref{cga},
\begin{equation}
\label{convae}
\lim_{E\to -\infty} \big[ A_E - \log (-E) \big] = 
\max_{r\in[\rho_-,\rho_+]} \log \chi(r) = A_\asy \;.
\end{equation}
Fix $\varphi\in C^{1+1}([-1,1])$ such that $\varphi(\pm
1)=\varphi_\pm$ and $\varphi' >0$. From \eqref{convae} it easily
follows that
\begin{equation}
\label{pezzo}
\lim_{E\to -\infty} \; \int_{-1}^1\!du \: \Big\{ 
\frac 1E \Big[ \varphi' \log \varphi' - (\varphi' - E) \log
(\varphi'  - E) \Big] - (A_E -A_\asy)  \Big\} = 0\;.
\end{equation}
Recalling \eqref{ssg}, \eqref{Grf} and \eqref{Grf-in}, from the
convexity of the real function $F: [0,1] \to \bb R$,
$F(\rho)= \rho \log \rho + (1-\rho)\log(1-\rho)$ and
\eqref{pezzo} we get
\begin{equation*}
\varliminf_{E\to -\infty} S_E(\rho_E) \ge 
\varliminf_{E\to -\infty} \mc G_E(\rho_E,\varphi) 
\ge \mc G_\asy (\rho,\varphi)\;.
\end{equation*}
The proof of the $\Gamma$-liminf inequality is now completed by
optimizing on $\varphi$. Note indeed that the supremum in \eqref{s-in}
can be restricted to strictly increasing $\varphi\in C^{1+1}([-1,1])$
such that $\varphi(\pm 1)=\varphi_\pm$.
\end{proof}

\begin{proof}[Proof of Theorem~\ref{t:gconv}: $\Gamma$-limsup
  inequality]
  
Fix $\rho\in\mc M$, we need to exhibit a sequence $\{\rho_E\}\subset
\mc M$ converging to $\rho$ in $\mc M$ as $E\to -\infty$ such that
$\varlimsup_E S_{E}(\rho_E) \le S_{\asy}(\rho)$. We claim that the
constant sequence $\rho_E=\rho$ meets this condition.

Recalling item (ii) in Theorem~\ref{t:S=V}, let $\phi_E:=
\Phi(\rho)\in \mc F_E$ be the solution to \eqref{Deq} in which we
indicated explicitly its dependence on $E$. From the concavity of the
real function $F: [0,\infty) \to \bb R$, $F(x)= E^{-1}\big[ x\log x
-(x-E)\log(x-E)\big]$, $E<0$, Jensen inequality, and \eqref{convae} we
deduce
\begin{equation}
\label{pezzo2}
\varlimsup_{E\to -\infty} \int_{-1}^1\!du\, 
\Big\{ \frac 1E \Big[ \phi_E' \log \phi_E' 
- (\phi_E' - E) \log (\phi_E'  - E) \Big] - (A_E -A_\asy) \Big\} 
\le 0 \;.
\end{equation}
Since $\mc F_E\subset \mc F_\asy$ and $\overline{\mc F}_\asy$ is
compact, the sequence $\{\phi_E\}$ is precompact in $\overline {\mc
  F_\asy}$. Let now $\phi^*\in \overline{\mc F}_\asy$ be any limit
point of $\{\phi_E\}$ and pick a subsequence $E'\to -\infty$ such that
$\phi_{E'}\to \phi^*$ in $\mc F_\asy$. In particular $\phi_{E'}(u)\to
\phi^*(u)$ Lebesgue a.e.  Recalling Theorem~\ref{t:S=V} (ii),
\eqref{Grf}, \eqref{Grf-in}, and using \eqref{pezzo2} we get that
\begin{equation*}
\varlimsup_{E'\to -\infty} S_{E'}(\rho) = 
\varlimsup_{E'\to -\infty} \mc G_{E'}(\rho,\phi_{E'}) 
\le \mc G_\asy (\rho,\phi^*)  \le S_\asy(\rho)\;,
\end{equation*}
which concludes the proof.
\end{proof}

\begin{proof}[Proof of Theorem~\ref{t:cof}]
Existence of a maximizer for \eqref{s-in} follows from the compactness
of $\overline{\mc F}_\asy$ and from the continuity of $\mc
G_\asy(\rho,\cdot)$ for the topology of $\overline{\mc F}_\asy$.  On
the other hand, the strict concavity of the function $F:
[\varphi_-,\varphi_+] \to \bb R_+$, $F(\varphi) = -
\log(1+e^\varphi)$, gives the uniqueness of the maximizer.

The proof of the convergence of the maximizers follows a variational
approach. Given $\rho\in\mc M$ and $E< 0$ we define $\upbar{\mc
  G}_E(\rho,\cdot): \overline{\mc F}_\asy \to [-\infty,+\infty)$ by
\begin{equation*}
\upbar{\mc G}_E(\rho,\varphi) :=
\begin{cases}
\mc G_E(\rho,\varphi) & \textrm{ if $\varphi\in\mc F_E$} \\
-\infty    & \textrm{ otherwise.}
\end{cases}
\end{equation*}
By \cite[Theorem 1.21]{Braides}, with all inequalities reversed since
we focus on maximizers instead of minimizers, the convergence of the
sequence $\{\phi_E\}$ to $\phi$ in $\overline{\mc F}_\asy$ follows
from the next three conditions. Fix $\rho\in\mc M$ and
$\varphi\in\overline{\mc F}_\asy$ then:
\begin{itemize}
\item[(i)]{for any sequence $\varphi_E\to\varphi$ in $\overline {\mc
      F_\asy}$, $\varlimsup_E \upbar{\mc G}_E(\rho,\varphi_E) \le {\mc
      G}_\asy(\rho,\varphi)$;}
\item[(ii)]{there exists a sequence $\varphi_E\to\varphi$ in
    $\overline{\mc F}_\asy$ such that $\varliminf_E \upbar{\mc
      G}_E(\rho,\varphi_E) \ge {\mc G}_\asy(\rho,\varphi)$;}
\item[(iii)]{$\phi$ is the unique maximizer for the functional ${\mc
      G}_\asy(\rho,\cdot)$ on $\overline{\mc F}_\asy$.}
\end{itemize}

\smallskip
\noindent\emph{Proof of (i).} We may assume that $\varphi_E\in \mc
F_E$; the proof of (i) is then achieved by noticing that
\eqref{pezzo2} holds also if $\phi_E$ is replaced by $\varphi_E$. 

\noindent\emph{Proof of (ii).} 
Assume firstly that $\varphi$ belongs to $C^1([-1,1])$ and satisfies
$\varphi(\pm 1)=\varphi_\pm$, $\varphi' > 0$. Since \eqref{pezzo}
holds for such $\varphi$, it is enough to take the constant sequence
$\varphi_E=\varphi$. The proof of (ii) is completed by a density
argument, see e.g.\ \cite[Rem.~1.29]{Braides}. More precisely, it is
enough to show that for each $\varphi\in\overline{\mc F}_\asy$ there
exists a sequence $\varphi^n\in C^1([-1,1])$ satisfying $\varphi^n(\pm
1)=\varphi_\pm$, $(\varphi^n)' > 0$ and such that
$\varphi^n\to\varphi$ in $\overline{\mc F}_\asy$, $\mc G_\asy
(\rho,\varphi^n) \to \mc G_\asy (\rho,\varphi)$. This is implied by
classical results on the approximation of $BV$ functions by smooth
ones.

As we have already shown (iii), the proof is completed. 
\end{proof}

\bigskip
\subsection*{Acknowledgments}

The results within this paper are a natural development of our
collaboration with A.\ De Sole and G. Jona-Lasinio to whom we are 
in a great debt. L.B.\ acknowledges the kind hospitality at IMPA and
the support of PRIN MIUR.

\end{document}